\documentclass[12pt]{amsart}
\usepackage{amsmath}
\usepackage{amssymb}

\usepackage{algorithm}
\usepackage{algpseudocode}
\usepackage{graphics}
\newcommand{\E}{\mathbb{E}}
\newcommand{\R}{\mathbb{R}}

\newcommand{\cC}{\mathcal{C}}

\newcommand{\im}{\operatorname{im}}

\newcommand{\tr}{\operatorname{Tr}}
\newcommand{\proj}{\operatorname{proj}}
\newcommand{\inte}{\operatorname{int}}

\newcommand{\cA}{\mathcal{A}}
\newcommand{\cT}{\mathcal{T}}
\newcommand{\Y}{\mathbb{Y}}
\newcommand{\SSS}{\mathbb{S}}
\newcommand{\norm}[1]{\left|\left|#1\right|\right|}
\newcommand{\abs}[1]{\left|#1\right|}

\newtheorem{thm}{Theorem}[section]
\newtheorem{lem}[thm]{Lemma}
\newtheorem{cor}[thm]{Corollary}
\newtheorem{dfn}[thm]{Definition}
\newtheorem{prop}[thm]{Proposition}

\theoremstyle{remark}
\newtheorem{examp}[thm]{Example} 

\newcommand{\iprod}[2]{\left\langle {#1}, {#2} \right\rangle}

\begin{document}

\title[Convex optimization via primal-dual metrics]{Interior-point algorithms
for convex optimization\\ based on primal-dual metrics}

\author{Tor Myklebust, Levent Tun\c{c}el}
\thanks{Tor Myklebust:
Department of Combinatorics and Optimization, Faculty
of Mathematics, University of Waterloo, Waterloo, Ontario N2L 3G1,
Canada (e-mail: tmyklebu@csclub.uwaterloo.ca).  Research of this author
was supported in part by an NSERC Doctoral Scholarship, ONR Research Grant
N00014-12-10049, and a Discovery Grant from NSERC.\\
Levent Tun\c{c}el: Department of Combinatorics and Optimization, Faculty
of Mathematics, University of Waterloo, Waterloo, Ontario N2L 3G1,
Canada (e-mail: ltuncel@uwaterloo.ca).
Research of this author
was supported in part by an ONR Research Grant
N00014-12-10049, and Discovery Grants from NSERC}

\date{November 7, 2014}

\begin{abstract}
We propose and analyse primal-dual interior-point algorithms for convex
optimization problems in conic form.
The families of algorithms we analyse are so-called short-step algorithms and
they match the current best iteration complexity bounds for primal-dual
symmetric interior-point algorithm of Nesterov and Todd, for symmetric cone
programming problems with given self-scaled barriers.
Our results apply to any self-concordant barrier for any convex cone.
We also prove that certain specializations of our algorithms to hyperbolic cone
programming problems (which lie strictly between symmetric cone programming and
general convex optimization problems in terms of generality) can take
advantage of the favourable special structure of hyperbolic barriers.
We make new connections to Riemannian geometry, integrals over operator spaces,
Gaussian quadrature, and strengthen the connection of our algorithms to
quasi-Newton updates and hence first-order methods in general.
\end{abstract}

\keywords{primal-dual interior-point methods, convex optimization, variable
metric methods, local metric, self-concordant barriers, Hessian metric,
polynomial-time complexity;\\
\emph{AMS subject classification (MSC):} 90C51, 90C22, 90C25, 90C60, 90C05, 65Y20, 52A41, 49M37, 90C30}

\maketitle

\newpage

  \begin{section}{Introduction}
    Convex optimization problems (of minimizing a given convex function in a
    given convex set) form a beautiful research area with very powerful theory
    and algorithms and many far-reaching applications.
    Among the main algorithms to solve convex optimization problems are modern
    interior-point methods.
    The modern theory of interior-point methods has flourished since
    Karmarkar's ground-breaking paper \cite{Karmarkar1984}.

    Every convex optimization problem can be paired with another convex
    optimization problem based on the same data, called its \emph{dual}.
    Rigorous solution methods for convex optimization problems typically
    generate, together with solutions to the problem at hand (primal problem),
    solutions for its dual.
    A particularly successful line of research pursued methods that work
    ``equally hard'' at solving both primal and dual problems simultaneously,
    called \emph{primal-dual symmetric methods} (for a rigorous definition, see
    \cite{Tuncel1998}).

    In the special cases of linear programming (LP) and semidefinite
    programming (SDP), these primal-dual symmetric methods, in addition to
    carrying certain elegance, led to improved results in theory, in
    computation, and in applications.

    Part of the success of primal-dual symmetric methods for LP and SDP might stem from the
    fact that both classes admit convex conic formulations where the underlying
    cone is \emph{self-dual} (the primal convex cone and the dual convex cone
    are linearly isomorphic) and
    \emph{homogeneous} (the automorphism group of the cone acts transitively in
    its interior).
    Convex cones that are both homogeneous and self-dual are
    called \emph{symmetric cones}.
    The success of primal-dual symmetric interior-point methods was further
    extended to the setting of symmetric cone programming, which led to deeper
    connections with other areas of mathematics.

    In this paper, we will extend many of the underlying
    mathematical entities, analyses and iteration complexity bounds of these
    algorithms to a general convex optimization setting (with arbitrary convex
    cones).

    The primal variable $x$ lives in a finite dimensional vector space $\E$ and
    the dual variables $y$ and $s$ live in the finite dimensional vector spaces
    $\Y$ and $\E^*$ respectively, where $\E^*$ is the dual space of $\E$.
    Every convex optimization problem can be put
    into the following conic form (under some mild assumptions):
    \[
      \begin{array}{ccrclcc}
       (P) & \mbox{inf} &  \langle c, x \rangle & & & &  \\
           & & {\cA}(x) & = & b , & & \\
           & & x & \in & K, & & \\
       \end{array}
    \]
    where ${\cA}: \E \to \Y^*$ is a linear map, $b \in \Y^*$, $c \in \E^*$ and
    $K \subset \E$ is a pointed, closed, convex cone with nonempty interior.
    We assume that $\cA, b, c$ are given explicitly and $K$ is described via $F
    : \inte(K) \to \R$, a \emph{logarithmically homogeneous self-concordant
    barrier function for $K$} (defined in the next section).  We assume
    without loss of generality that
    $\cA$ is surjective (i.e., $\cA(\E) = \Y^*$).

    We define the dual of $(P)$ as
    \[
      \begin{array}{ccrclclcc}
       (D) & \mbox{sup} &  \langle b, y \rangle_D & & & & & &  \\
           & & {\cA}^*(y) & + & s & = & c , & & \\
           & & & & s & \in & K^*, & & \\
       \end{array}
    \]
    where $K^*$ is the dual of cone $K$, namely
    \[
    K^* \,\, := \,\, \left\{ s \in \E^*: \,\, \langle s, x \rangle \,\, \geq
    \,\, 0, \,\,
    \forall \, x \in K \right\}.
    \]
    We are using $\langle \cdot, \cdot \rangle$ to denote the dual pairing
    on $(\E, \E^*)$ and $\langle \cdot, \cdot \rangle_D$ to denote the dual
    pairing on $(\Y,\Y^*)$.
    $\cA^*: \Y \to \E^*$ denotes the adjoint of $\cA$ defined by the equations:
    \[
    \iprod{\cA^*(y)}{x} = \iprod{{\cA}(x)}{y}_D,
    \,\,\,\, \forall \,\, x \in \E, \,\, y \in \Y.
    \]

In this paper, we utilize many of the fundamental techniques developed
throughout the history of interior-point methods.  One of the main algorithms
we design and analyse is a predictor-corrector algorithm generalizing the
algorithm of Mizuno, Todd and Ye \cite{MTY1993} from the LP setting.
However, even in the LP setting, some of our algorithms are
new.  Our algorithms use Newton directions as in Renegar's algorithm for LP
\cite{Renegar1988}, and one of our main algorithms uses a
similar predictor-corrector scheme, but both predictor and corrector
parts in a primal-dual setting.

The modern theory of interior-point methods employed the concept of
self-concordant barrier functions (see Nesterov and Nemirovski \cite{NN1994}).
The Hessians of these nice convex functions induce local metrics with excellent
properties.  For instance, the unit ball induced by the Hessian at a point,
called the \emph{Dikin ellipsoid}, is contained in the cone.

The self-concordance property implies that, as we make local moves in the domain
of the barrier function, the local metrics do not change fast unless their
norms are large.
More precisely, the speed of the change in the local metric can be bounded in terms of the local metric itself.

One of the indicators of how good a barrier function is (in terms of the local
metrics it generates) can be measured by how well the Dikin ellipsoid
approximates the domain of the function.
This leads to the notion of \emph{long-step Hessian estimation property}
(defined in Section~\ref{sec:prelim}) of barriers.
This property amounts to extending the controlled change of the Hessian of the
barrier to the ``best, norm-like local approximation'' of the domain of the
barrier.
This long-step Hessian estimation property has been proven for self-scaled
barriers \cite{NT1997,NT1998} (whose domains are symmetric cones and these barriers
have been completely classified \cite{HauserGuler2002,HauserLim2002,Schmieta2000}) and
hyperbolic barriers \cite{Guler1997} (whose domains are hyperbolicity cones;
for results on the structure of these cones see
\cite{Guler1997,BGLS2001,Renegar2006,Renegar2013}), but there exist barriers for which
this property fails.

For a related, but weaker notion of long-step, also see \cite{NN1998}.  Indeed,
we would like our algorithms to exploit such properties when they are present.
(Beyond the set-up of symmetric cones and self-scaled barriers, exploitation of
this long step property forces us to break the primal-dual symmetry of our
algorithms in the favor of the problem with barrier function admitting this
property.)  Our general approach and many of our main technical results are
primal-dual symmetric (in the sense of \cite{Tuncel1998}); however, beyond
symmetric cones and self-scaled barriers, there may be advantages to breaking
the primal-dual symmetry in favour of the better behaved (or better posed)
problem.  Our approach also provides useful tools for exploiting such
structures when they are present.

Independent of this work, recently, simultaneously with an announcement of this
work \cite{Myklebust-Tuncel-Talk-2013}, Renegar announced a primal-dual
affine-scaling method for hyperbolic cone programming \cite{Renegar-Talk-2013}
(also see Renegar and Sondjaja \cite{RS2014}). Nesterov and Todd
\cite{NT1997,NT1998} present very elegant primal-dual
interior-point algorithms with outstanding mathematical properties in the
setting of self-scaled barriers; however, beyond self-scaled barriers, there
are many other primal-dual approaches that are not as symmetric, but retain
some of the desired properties
(see \cite{Tuncel2001,NemTun2005,Chua2007,Chua2009,NesTun2009,Nesterov2012}).

Some recent study of the interior-point methods and the central paths defined
by self-concordant barriers led to connections with Riemannian geometry,
see \cite{NT2002,NN2008}.  We also make some additional connections
to Riemannian geometry through the local primal-dual metrics that we utilize
in this paper.

    Hessians of self-concordant barriers, in addition to inducing local
    metrics, provide linear isomorphisms between primal and dual spaces $\E$
    and $\E^*$.  In the special case of self-scaled barriers, we have $F''(w)
    \inte(K) = \inte(K^*)$ for every $w \in \inte(K)$.  We focus on
    generalization of such behaviour and simulate it with a self-adjoint,
    positive-definite map $T^2$ mapping $\E^*$ to $\E$.  This leads to (via its
    unique self-adjoint, positive-definite square-root $T$) construction of a
    \emph{$v$-space} as a space that is ``half-way between'' $\E$ and $\E^*$,
    i.e., $\E^{(*/2)}$.

    The overall structure of the remainder of this paper is as follows:
    \begin{itemize}
      \item Section \ref{sec:prelim} presents some fundamental notation, definitions and
      properties of underlying primal-dual spaces and the class of convex barrier functions.
      \item Section \ref{sec:intscal} presents a new primal-dual scaling map
        based on integration of the barrier's Hessian.
      \item Section \ref{sec:dfp} presents some low-rank update formulae for the construction
      of local primal-dual metrics and connects these formulae to classical work on quasi-Newton
      updates.
      \item Section \ref{sec:pd-symmetry} addresses primal-dual symmetry of our approach via geodesic convexity
      of the underlying sets of local metrics.
      \item Section \ref{sec:scaling-point} delves deeper into investigating the relationship between the
      set of Hessians of self-concordant barriers and the set of local primal-dual metrics we use.
      \item Sections \ref{sec:analysis} and \ref{sec:algorithms} combine a crude approximation of the
        scaling from Section \ref{sec:dfp} with the technique of Section 3 to
        derive and analyse some theoretically efficient interior-point algorithms for convex
        programming.
      \item Section \ref{sec:hyp} examines some of the special properties of our approach for hyperbolic
        cones that may make our techniques particularly effective.
    \end{itemize}
  \end{section}

  \begin{section}{Preliminaries}\label{sec:prelim}
  In this section, we introduce some of the fundamental concepts,
  definitions and properties that will be useful in the rest of the paper.
  For a more detailed exposure to these concepts, see the standard
  references \cite{NN1994,NT1997,NT1998,Renegar2001,Nesterov2004}
  as well as \cite{Tuncel1998,Tuncel2001}.
    \begin{dfn}\label{defi:self-concordance}
      (Self-concordance)
      Let $F: \inte(K) \rightarrow \R$ be a $\cC^3$-smooth convex
      function such that $F$ is a barrier for $K$ (i.e. for every sequence in
      the interior of $K$, converging to a boundary
      point of $K$, the corresponding function values $F(x)\to +\infty$)
      and there exists $\vartheta\geq1$ such that, for each $t> 0$,
      $$F(tx) \,\, = \,\, F(x) - \vartheta \ln(t),$$
      and
      \begin{eqnarray}\label{eq:D3}
        |D^3F(x)[h,h,h]| & \leq & 2 \left(D^2F(x)[h,h]\right)^{3/2}
      \end{eqnarray}
      for all $x \in \inte(K)$ and for all $h\in \E$.
      Then $F$ is called a {\em $\vartheta$-logarithmically homogeneous
      self-concordant barrier
      ($\vartheta$-LHSCB)} for $K$.
    \end{dfn}

    If $F$ is a $\vartheta$-LHSCB for $K$, then its (modified) Legendre-Fenchel
    conjugate
    \[
      F_*(s) \,\, := \,\, \sup\{-\langle s, x \rangle - F(x):
         \,\,\, x \in \inte(K)\},
    \]
    is a $\vartheta$-LHSCB for the dual cone $K^*$ (Nesterov and Nemirovskii
    \cite{NN1994}).  We refer to $F_*$ simply as the \emph{conjugate barrier}.

Once we have a $\vartheta$-LHSCB $F$ for $K$, at every point $x \in \inte(K)$,
the Hessian of $F$ defines a local metric.  For every $h \in \E$
the \emph{local norm induced by $F$ at $x$} is
\[
\norm{h}_x \,\, := \,\, \langle F^{''}(x) h, h \rangle^{1/2}.
\]
It is not hard to see from the definition of LHSCBs
that $F''(x): \E \to \E^*$ is self-adjoint and positive-definite,
$\left[F''(x)\right]^{-1}: \E^* \to \E$ is well-defined and is self-adjoint
as well as positive-definite.
For every $u \in \E^*$, we define
\[
\norm{u}_x^*:= {\iprod{u}{\left[F''(x)\right]^{-1} u}}^{1/2}.
\]
\begin{prop}
\label{prop:2.1}
Let $F$ be a $\vartheta$-LHSCB for $K$.  Then for every $x \in \inte(K)$,
$\norm{\cdot}_x^*$ is the norm dual to $\norm{\cdot}_x$.  I.e., for
every $x \in \inte(K)$,
\[
\norm{u}_x^* = \sup \left\{\iprod{u}{h}: \norm{h}_x \leq 1, h \in \E \right\},
\,\,\,\, \forall u \in \E^*.
\]
\end{prop}

The above proposition implies:
\[
\abs{\iprod{u}{h}} \leq \norm{u}_x^* \norm{h}_x,
\,\,\,\,
\forall u \in \E^*, h \in \E.
\]
We use the above ``Cauchy-Schwarz'' inequality quite often.
Note that
\[
\norm{F'(x)}_x^* = \norm{h}_x,
\]
where $h:= \left[F''(x)\right]^{-1} F'(x)$, the Newton step at $x$ for minimizing $F$.

    \begin{thm}\label{thm:DikinEllipsoid1}
      Let $F$ be a $\vartheta$-LHSCB for $K$.  Then, for every $x \in
      \inte(K)$, the open unit ellipsoid centered at $x$ and defined by the
      positive-definite Hessian $F{''}(x)$ is contained in the interior of the
      cone.  That is
      \[
        E(x; F^{''}(x)) \,\, := \,\, \{
        z \in \E: \,\, \langle F^{''}(x)(z-x), z-x
        \rangle \, < \, 1 \}
        \, \subset \inte(K).
      \]
      Moreover, for every $z \in \inte(K)$ such that
      $\alpha \, := \, \norm{x-z}_x < 1$, we have
      \[
        (1-\alpha)^2 F''(x)[h,h] \,\,
        \leq \,\, F''(z)[h,h] \,\, \leq \,\,
        \frac{1}{(1-\alpha)^2} F''(x)[h,h],
      \]
      for all $h \in \E$.
    \end{thm}

We use (as above) $F''(x)[h^{(1)},h^{(2)}]$ to denote the second derivative
of $F$ evaluated along the directions $h^{(1)},h^{(2)} \in \E$.  We also
use the notation $\iprod{F''(x) h^{(1)}}{h^{(2)}}$ to denote the same quantity.
In both expressions, $F''(x)$ is the Hessian of $F$.  As we deal with the
Hessians and other self-adjoint transformations in the same space, we sometimes
utilize the {L\"{o}wner order}, we write $A \preceq B$ to mean $(B-A)$ is
self-adjoint and positive semidefinite.
With these clarifications, the above inequalities in the statement of the
last theorem, can be equivalently written as:
\[
(1-\alpha)^2 F''(x) \,\,
        \preceq \,\, F''(z) \,\, \preceq \,\,
        \frac{1}{(1-\alpha)^2} F''(x);
\]
we refer to the above relations as the \emph{Dikin ellipsoid bound}.

For every $x \in \inte(K)$ and every $h \in \mathbb{R}^n$,
define
\[
\sigma_x(h) \,\, := \,\,
\frac{1}{\sup \left\{t: \,\, (x - t h) \in K \right\}}.
\]
We say that \emph{$F$ has the long-step Hessian estimation property}
if
\begin{eqnarray}
\label{eq:long-stepHessEst}
\frac{1}{\left[1 + t\sigma_x(-h)\right]^2} F^{''}(x)
& \preceq & F^{''}(x-th)
\,\, \preceq \,\, \frac{1}{\left[1 - t\sigma_x(h)\right]^2} F^{''}(x),
\end{eqnarray}
for every $x \in \inte(K)$, $h \in \mathbb{R}^n$ and
$t \in \left[0, 1/\sigma_x(h)\right).$
Nesterov and Todd \cite{NT1997} proved that every self-scaled barrier
has this property.
G{\"u}ler \cite{Guler1997} extended this property
to {\em hyperbolic barriers}. However, Nesterov proved (see Theorem 7.2 in \cite{Guler1997})
that the relation (\ref{eq:long-stepHessEst})
can hold for a self-concordant barrier
and its conjugate only if $K$ is a symmetric cone.  Essentially equivalent
properties are expressed as the \emph{convexity of $\iprod{-F(x)}{u},$}
for every $u \in K$, or, as $F$ having \emph{negative curvature}: $F'''(x)[u] \preceq 0$
for every $x \in \inte(K)$ and for every $u \in K$.

All of the properties listed in the next theorem can be derived
directly from the logarithmic homogeneity property of $F$.

\begin{thm}\label{thm:LHSCB-simple}
  Let $F$ be a $\vartheta$-LHSCB barrier for $K$.Then
  for all $x \in \inte(K)$ and $s \in \inte(K^*)$,
  $F$ has the following properties:
  \begin{enumerate}
    \item
      For all $k \geq 1$ integer
      and $t > 0$, if $F$ is $k$ times differentiable, then
      $D^{k}F(tx) = \frac{1}{t^k} D^{k}F(x);$
    \item
      $\langle -F^{'}(x), x\rangle = \vartheta;$
    \item
      $F^{''}(x)x= -F^{'}(x);$
    \item
      $\iprod{F''(x)x}{x} = \vartheta = \norm{x}_x^2,
      \,\,\,\, \iprod{\left[F''(x)\right]^{-1}F'(x)}{F'(x)}
      = \vartheta = \left(\norm{F'(x)}_x^*\right)^2;$
    \item
      $F^{'''}(x)[x] = -2F^{''}(x).$
  \end{enumerate}
\end{thm}

The LHSCB $F$ and its conjugate barrier $F_*$ interact very nicely due to the
elegant and powerful analytic structure imposed by the Legendre-Fenchel
conjugacy:

\begin{thm}
  \label{thm:Fprime}
  Let $F$ be a $\vartheta$-LHSCB barrier for $K$.Then
  for all $x \in \inte(K)$ and $s \in \inte(K^*)$,
  $F$ and $F_*$ satisfy the following properties:
  \begin{enumerate}
    \item
      $F_*(-F^{'}(x)) = -\vartheta -F(x)$
      and $F(-F_*^{'}(s)) = -\vartheta -F_*(s);$
    \item
      $-F_*^{'}(-F^{'}(x)) = x$ and
      $-F^{'}(-F_*^{'}(s)) = s;$
    \item
      $F_*^{''}(-F^{'}(x)) = \left[F^{''}(x)\right]^{-1}$ and
      $F^{''}(-F_*^{'}(s)) = \left[F_*^{''}(s)\right]^{-1}.$
  \end{enumerate}
\end{thm}

Maps between primal space $\E$ and the dual space $\E^*$ play very
fundamental roles.  Among such maps, those which take $\inte(K)$
to $\inte(K^*)$ bijectively are particularly useful.  One such map
is $-F'(\cdot)$ (i.e., the \emph{duality mapping}).  In the special
case of symmetric cones and self-scaled barriers $F$, $F''(w)$ gives
a linear isomorphism between $\inte(K)$ and $\inte(K^*)$ for every
choice of $w \in \inte(K)$.

    Fixing some inner product on $\E$, and hence some linear isomorphism
    between $\E$ and $\E^*$, we construct and study certain self-adjoint,
    positive-definite linear transformations mapping $\E^*$ to $\E$.
    Sometimes, it is useful for the sake of clarity in implementations of the
    upcoming algorithms to denote the dimension of $\E$ by $n$ and consider the
    matrix representations of such self-adjoint positive-definite linear
    transformations as elements of $\SSS^n_{++}$.
    Next we define three sets of such linear transformations: $\cT_0$, $\cT_1$,
    $\cT_2$ \cite{Tuncel2001}.  These sets are indexed based on their usage of
    information from the underlying self-concordant barrier functions
    (zeroth-order information only, up to first-order information only, and up
    to second-order information only, respectively).
    \begin{dfn}\label{defi:T0}
      For every pair $(x,s) \in \inte(K) \oplus \inte(K^*)$,
      we define
      \[
      \cT_0 (x,s) \,\, := \,\, \left\{ T \in \SSS_{++}^n: \,\,
      T^2(s) = x \right\}.
      \]
      In words, $\cT_0 (x,s)$ is the set of all positive-definite,
      self-adjoint, linear operators on ${\R}^n$ whose
      squares map $s$ to $x$.
    \end{dfn}

    For a given pair of $x,s$, let us define
    \[
      \mu:= \frac{\langle s, x \rangle}{\vartheta}.
    \]
    Note that if $x$ and $s$ are feasible in their respective
    problems, then $\vartheta \mu$ is their duality gap:
    \[
    \vartheta \mu =\iprod{s}{x} = \iprod{c}{x}-{\iprod{b}{y}}_D \geq 0.
    \]
    Moreover, upon defining $v:= T(s)=T^{-1}(x)$, we observe
    \[
      \vartheta \mu = \langle s, x \rangle = \langle T(s), T^{-1}(x) \rangle =
      \langle v, v \rangle.
    \]
    These linear transformations allow generalizations of so-called
    \emph{$v$-space} approaches to primal-dual interior-point methods from LP
    and SDP (see for instance,
    \cite{MAR1990,KMNY1991,JRT1996,SZ1996,Nesterov2008}) as well as symmetric
    cone programming settings \cite{NT1997,NT1998,Tuncel1998,Hauser2004} to the
    general convex optimization setting \cite{Tuncel1998,Tuncel2001}.  What
    perhaps started as a convenience of notation as well as simplicity and
    elegance of analysis in the 1980's, by now turned into a solid theoretical
    framework in which deep problems can be posed and solved whether they
    relate to the mathematical structures or are directly motivated by a family
    of algorithms.

    Note that the set $\cT_0$ does not use any information about the most
    difficult constraints of the convex optimization problem.  Indeed, in the
    worst-case, certain algorithms using only $\cT_0$ need not even converge to
    an optimal solution (see \cite{Wei2002}).  Using first order information
    via $F'(x)$ and $F_*'(s)$, we can focus on a smaller, but more interesting
    and useful subset of $\cT_0$:

    \begin{dfn}\label{defi:T1}
      For every pair $(x,s) \in \inte (K) \oplus \inte (K^*)$,
      we define
      \[
        {\cT}_1 (x,s) \,\, := \,\, \left\{ T \in \SSS_{++}^n: \,\,
        T^2 (s) = x, \,\, T^2(-F^{'}(x)) = -F_*^{'}(s) \right\}.
      \]
    \end{dfn}

    For convenience, we sometimes write ${\tilde x} := -F^{'}_*(s)$
    and ${\tilde s} := -F^{'}(x)$. One can think of ${\tilde x}$ and
    ${\tilde s}$
    as the {\em shadow} iterates, as ${\tilde x} \in \inte(K)$
    and ${\tilde s} \in \inte(K^*)$ and if
    $(x,s)$ is a feasible pair, then $\mu {\tilde x} = x$ iff
    $\mu {\tilde s} = s$ iff $(x,s)$ lies on the central path.
    We also denote ${\tilde \mu} :=
    \langle {\tilde x}, {\tilde s} \rangle /\vartheta.$
    In this context, $\tilde{x}$ is \emph{the primal shadow of the dual solution
    $s$}.  Analogous to the definition of $v$, we can now define
    \[
      w := T(\tilde{s})=T^{-1}(\tilde{x}).
    \]
    Note that $\langle w, w \rangle = \vartheta \tilde{\mu}.$

    Once we have such a transformation $T$, we can map the primal space
    using $T^{-1}$, and the dual space using $T$
    to arrive at the data of the underlying (equivalent)
    {\em scaled} primal-dual pair.
    We define ${\bar {\cA}} := {\cA}\left(T(\cdot)\right)$
    and, we have
    \[
      \begin{array}{cccclcr}
       ({\bar P}) & \mbox{inf} &  \langle Tc, \bar{x} \rangle & & & &  \\
           & & \bar{{\cA}}(\bar{x}) & = & b , & & \\
           & & \,\,\,\,\,\,\,\, \bar{x} & \in &T^{-1}(K), & & \\
       \end{array}
    \]

    \[
      \begin{array}{cccclcr}
       ({\bar D}) & \mbox{sup} &  \langle b, y \rangle_D & & & &  \\
           & & \bar{{\cA}}^*(y) +  \bar{s} & = & Tc, & & \\
           & & \,\,\,\,\,\,\,\,\,\,\,\,\,\,\,\,\,\,\,\, \bar{s} & \in & T(K^*). & & \\
       \end{array}
    \]
    Then the search directions for the scaled primal-dual pair
    $({\bar P})$ and $({\bar D})$ are
    (respectively)
    \[
      {\bar d}_x := T^{-1}(d_x), \,\,\, {\bar d}_s := T(d_s),
    \]
    where $d_x, d_s$ denote the search directions in the
    original primal and dual spaces respectively.
    We also define
    \[
      v \,\, := \,\, T(s) = T^{-1}(x).
    \]
    Note that $\langle v, v \rangle = \langle x, s \rangle = \vartheta \mu$.
      \begin{lem}\label{lem:2.1}
      (Nesterov and Todd \cite{NT1998}, also see \cite{Tuncel2001,Tuncel2010})
      For every $(x,s) \in \inte (K) \oplus \inte (K^*)$,
      \[
      \mu {\tilde \mu} \,\, \geq \,\, 1.
      \]
      Equality holds above iff $x = -\mu F_*^{'}(s)$
      (and hence $s = -\mu F^{'}(x)$).
    \end{lem}

Note that the equality above together with feasibility of $x$ and $s$
in their respective problems, define the \emph{primal-dual central path}:
\[
\left\{\left(x(\mu), s(\mu)\right): \mu > 0 \right\}.
\]
Further note that for each $T \in \cT_1$, we can equivalently express
the centrality condition as $v=\mu w$ (of course, together with the
requirement that the underlying $x$ and $s$ be feasible in their respective
problems).  Primal and dual deviations from the central path are:
\[
      \delta_P := x - \mu \tilde{x} \mbox{ and } \delta_D := s - \mu \tilde{s}.
\]
In $v$-space, these deviations are both represented by
\[
\delta_v := v -\mu w.
\]
    \begin{cor}\label{cor:2.1}
      For every $(x,s) \in \inte (K) \oplus \inte (K^*)$,
      \[
      \langle s + \mu F^{'}(x), x+ \mu F_*^{'}(s) \rangle
      = \langle \delta_D, \delta_P \rangle
      = \norm{\delta_v}^2
      \,\, \geq \,\, 0.
      \]
      The equality holds above iff $x = -\mu F_*^{'}(s)$
      (and hence $s = -\mu F^{'}(x)$).
    \end{cor}

The above quantities have their counterpart, called the \emph{gradient
proximity measure} and denoted by $\gamma_G(x,s)$, in the work of Nesterov
and Todd \cite{NT1998}.  In their paper, the corresponding measure is
\[
\gamma_G(x,s) = \vartheta \left(\mu\tilde{\mu} -1 \right) = \frac{\norm{\delta_v}^2}{\mu}.
\]
    Now, we consider the system of linear equations:
    \begin{eqnarray}
    \label{eq:LE01}
    {\bar {\cA}} ({\bar d}_x) & = & 0 \\
    \label{eq:LE02}
    {\bar {\cA}}^*(d_y) + {\bar d}_s& = & 0.\\
    \label{eq:LE03}
    {\bar d}_x+ {\bar d}_s& = & -v + \gamma \mu w,
    \end{eqnarray}
    where $\gamma \in [0,1]$ is a centering parameter.

    Clearly, $({\bar d}_x, {\bar d}_s)$ solves the above linear system of
    equations iff ${\bar d}_x$ and ${\bar d}_s$ are the orthogonal projections of $(-v +\gamma \mu
    w)$ onto the kernel of ${\bar {\cA}}$ and the image of ${\bar {\cA}}^*$
    respectively.  Given a pair of search directions, we define
    \[
    x(\alpha) := x + \alpha d_x, \,\,\,\, s(\alpha) := s + \alpha d_s.
    \]
Using the above observations, the following result is immediate.
    \begin{lem}\label{lem:2.2}
      Let $\gamma \in [0,1]$. Then
      \[
      \langle x(\alpha), s(\alpha)\rangle
      \,\, = \,\, \left[1-\alpha(1-\gamma)\right] \langle x, s\rangle.
      \]
    \end{lem}

Now, we are ready to define the set of local primal-dual metrics which
will utilize local second-order information from the underlying pair
of self-concordant barriers.
    For each $F$ and each pair $(x,s) \in \inte(K) \oplus \inte(K^*)$,
    consider the optimization problem (an SDP when considered in terms of variables $T^2$ and $\xi$)
    with variables $T \in \SSS^n$ and $\xi \in \R$:
    \[
      \begin{array}{ccccccc}
       (PD)^2 & \mbox{inf} &  \xi & & & &  \\
           & & T^2(s) & = & x, & & \\
           & & T^2(-F^{'}(x)) & = & -F^{'}_*(s), & & \\
           & & \frac{\mu}{\xi \left[\vartheta(\mu\tilde{\mu}-1)+1\right]}
    F^{''}_*(s)
    & \preceq & T^2 & \preceq
    & \frac{\xi \left[\vartheta(\mu\tilde{\mu}-1)+1\right]}{\mu}
    \left[F^{''}(x)\right]^{-1}, \\
           & & \,\,\,\,\,\,\,\, \xi \,\, \geq \,\, 1, & & & &T \in \SSS^n. \\
       \end{array}
    \]

    Let $\xi^*$ be the optimum value of $(PD)^2$.

    \begin{dfn}\label{defi:2.1}
      For every pair $(x,s) \in \inte(K) \oplus \inte(K^*)$,
      we define
      \[
      {\cT}_2 (\eta; x,s) \,\, := \,\, \left\{ T \in \SSS^n_{++}: \,\,
      \begin{array}{l}
      T^2 (s) = x, \\
      T^2(-F^{'}(x)) = -F_*^{'}(s), \\
      \frac{\mu}{\eta \xi^* \left[\vartheta(\mu\tilde{\mu}-1)+1\right]}
      F^{''}_*(s)
      \,\, \preceq \,\, T^2 \,\, \preceq \,\,
      \frac{\eta \xi^* \left[\vartheta(\mu\tilde{\mu}-1)+1\right]}{\mu}
      \left[F^{''}(x)\right]^{-1}
      \end{array}
      \right\},
      \]
      for $\eta \geq 1$.
    \end{dfn}

Sometimes we find it convenient to refer to $T^2$ directly. So, we define
$\cT_0^2$, $\cT_1^2$, $\cT_2^2$ as the set of $T^2$ whose self-adjoint, positive
definite square-root $T$ lie in
$\cT_0$, $\cT_1$, $\cT_2$ respectively.

When the underlying cone $K$ is symmetric, we have $(P)$ and $(D)$ as
symmetric cone programming problems.  Symmetric cones (are homogeneous
and self-dual, and) admit self-scaled barriers.  For such problems,
there is a very small upper bound on $\xi^*$ for every pair of interior-points.

    \begin{thm}\label{thm:2.2}
      (Nesterov and Todd \cite{NT1997,NT1998}, also see \cite{Tuncel2001})
      Let $K$ be a symmetric cone and
      $F$ be a self-scaled barrier for $K$.  Then, for every
      $(x, s) \in \inte (K) \oplus \inte (K^*)$,
      $\xi^* \leq 4/3$.
    \end{thm}
Being able to establish a nice upper bound on $\xi^*$ for all iterates, and
then, for some small $\eta \geq 1$, finding an efficient way of computing
and element of $\cT_2(\eta;x,s)$, directly yield primal-dual interior-point
algorithms with good properties.  In particular, if the upper bounds on
$\xi^*$ and $\eta$ are both $O(1)$, then such results directly yield
primal-dual interior-point algorithms with iteration complexity bound
$O\left(\sqrt{\vartheta}\ln\left(1/\epsilon\right)\right)$; see, \cite{Tuncel2001}.
Next, we consider various ways of constructing good linear transformations
in $\cT_1$ and $\cT_2$.
  \end{section}

  \begin{section}{Primal-dual metrics via Hessian integration}\label{sec:intscal}

In the special case where cone $K$ is symmetric and $F$ is a self-scaled
barrier for $K$, Nesterov and Todd \cite{NT1997,NT1998} identified a specific
element of the sets $\cT_0$, $\cT_1$ and $\cT_2$ in terms of the Hessians of
certain \emph{scaling points} (i.e., for every pair $(x,s)$, there exists $w
\in \inte(K^*)$ such that $T^2=F_*''(w)$).  There is also an explicit linear
algebraic formula which expresses the Hessian at $w$ in terms of the Hessians
of $F$ and the conjugate barrier $F_*$ at $x$ and $s$ respectively.  The next
theorem achieves an analogous goal in the fully general case of an arbitrary
convex cone $K$ and an arbitrary self-concordant barrier $F$ for $K$ by
expressing a primal-dual scaling in $\cT_1^2$ as an integral of Hessians along
the line segment joining the dual iterate to the dual shadow of the primal
iterate.  Later in Section \ref{sec:scaling-point}, we prove that beyond
symmetric cones and self-scaled barriers, one should \emph{not} expect in
general to find, for every pair of primal dual interior points $(x,s)$, a
$w \in \inte(K^*)$ such that $F_*''(w) \in \cT^2(x,s)$.
Perhaps surprisingly, we prove next that ``an average of the Hessians''
works!

    \begin{thm}\label{thm:3.1}
      Let $F$ be a LHSCB for $K$ and $(x,s) \in \inte(K) \oplus \inte(K^*)$.
      Then, the linear transformation
      \[
        T_D^2 := \mu \int_0^1 F''_*(s - t \delta_D) dt
      \]
      is self-adjoint, positive-definite, maps $s$ to $x$, and
      maps $\tilde{s}$ to $\tilde{x}$.  Therefore, its unique self-adjoint,
      positive-definite square root $T_D$ is in $\cT_1(x,s)$.
    \end{thm}
    \begin{proof}
      Using the fundamental theorem of calculus (for the second equation below)
      followed by the property
      $-F'_*\left(-F'(x)\right)=x$ (for the third equation below), we obtain
      $$
          T_D^2 \delta_D
        = \mu \int_0^1 F''_*(s - t \delta_D) \delta_D dt
        = \mu \left(F'_*(s - \delta_D)-F'_*(s)\right)
        = \mu \left({x}/\mu - \tilde{x}\right)
        = \delta_P.
      $$
      We next compute, using the substitution $\tilde{t} = 1/t$,
      \begin{eqnarray*}
        T_D^2 s &=& \mu \int_0^1 F''_*(s - t \delta_D) s dt\\
                &=& \mu \int_0^1 \frac{1}{t^2} F''_*(s/t - \delta_D) s dt\\
                &=& \mu \int_1^\infty F''_*(\tilde{t} s - \delta_D) s d\tilde{t}\\
                &=& -\mu F'(s - \delta_D) = x.
      \end{eqnarray*}
      Further, $T_D^2$ is the mean of some self-adjoint, positive-definite
      linear transformations, so $T_D^2$ itself is self-adjoint and
      positive-definite.
    \end{proof}

    We call this scaling operator the \emph{dual integral scaling}.
    Note that the above theorem holds under weaker assumptions (we only
    used the facts that $F$ is logarithmically homogeneous and $F, F_*$ are
    {Legendre-type functions} in the sense of Rockafellar \cite{Rockafellar1970}).
    The dual integral scaling is expected to inherit many of the nice properties of
    the Hessians.  Thus, if $F_*$ is well-behaved, then one can prove nice
    bounds on the deviation of dual integral scaling from the dual Hessian at $s$ and
    $\mu \tilde{s}$:

    \begin{thm}
    \label{thm:3.2}
      If $\sigma < 1$ is such that, for any $t \in [0,1]$,
      $$
        (1-t \sigma)^2 F''_*(s)\preceq F''_*(s-t\delta_D)
        \preceq \frac{1}{(1 - t \sigma)^2} F''_*(s),
      $$
      then
      $$
       (1 - \sigma) \mu F''_*(s) \preceq T_D^2 \preceq \frac{1}{1-\sigma} \mu F''_*(s).
      $$
    \end{thm}
    \begin{proof}
      This follows directly from the definition of $T_D^2$.
    \end{proof}

    Interestingly, the dual integral scaling (the mean of Hessians along the
    line segment joining $s$ and $\mu \tilde{s}$) is not as ``canonical''
    as the Nesterov--Todd scaling (the geodesic mean of the Hessians
    joining the same two points in the interior of $K^*$)
    in terms of primal-dual symmetry
    properties.  For the remainder of this section, we
    elaborate on this and related issues.  Also, in Section \ref{sec:hyp}
    when we specialize on \emph{Hyperbolic Cone Programming} problems,
    we show that the integral scaling can have advantages (when one
    of the primal and dual problems is \emph{more tractable} or
    \emph{nicer} for the approach at hand, as a benefit of
    breaking this primal-dual symmetry well, and properties like
    those given by Theorem~\ref{thm:3.2}).

    Notice that the dual integral scaling
    $$
      \int_0^1 \mu F''_*(t s + (1-t) \mu \tilde{s}) dt
    $$
    and the primal integral scaling
    $$
      T_P^2 := \left(\int_0^1 \mu F''(t x + (1-t) \mu \tilde{x}) dt\right)^{-1}
    $$
    are both scalings that map $s$ to $x$ and $\tilde{s}$ to $\tilde{x}$.
    These are not in general the same, though they do coincide with the usual
    scaling $XS^{-1}$ in the case of linear programming.

%

    \begin{examp}{(A comparison of primal-dual local metrics for the
        positive semidefinite cone)}\\
      We work out the integral scaling for the positive
      semidefinite cone $\SSS^n_+$.  If $X$ is the primal iterate and $\tilde{X}$ is the
      primal shadow of the dual iterate, then we see that
      $$
        T_D^2[H,H] = \mu \int_0^1 \left\langle
            (t X + (1-t) \mu \tilde{X})^{-1} H
            (t X + (1-t) \mu \tilde{X})^{-1}, H
            \right\rangle dt.
      $$
      One can make this slightly more explicit.  There always exists
      a $U \in GL(n)$ such that $U X U^\top = I$ and $U \mu \tilde{X} U^\top$ is
      diagonal; one can compose a $Q$ that orthogonally diagonalises $X$, an $S$
      that scales $Q^\top X Q$ to the identity matrix, and a $Q'$ that
      orthogonally diagonalises $S Q^\top \mu \tilde{X} Q S$.  Say
      $U \mu \tilde{X} U^\top = D$.  Then we can compute
      $$
        T_D^2[U H U^\top, U H U^\top] = \mu \int_0^1 \left\langle
            (t I + (1-t) D)^{-1} H (t I + (1-t) D)^{-1}, H
            \right\rangle dt.
      $$
      In particular, if $H$ is $E_{ij}$, we have
      $$
        T_D^2[U E_{ij} U^\top, U E_{ij} U^\top] = \mu \int_0^1
          (t + (1-t) D_i)^{-1}
          (t + (1-t) D_j)^{-1} dt
        = \mu \frac{\ln D_j - \ln D_i}{D_j - D_i}.
      $$
      Special attention needs to be given to the case when $D_i = D_j$;
      here, the integral evaluates to $\mu/D_i$.

      If $H = E_{ij} + E_{kl}$, we have
      $T_D^2[U (E_{ij} + E_{kl}) U^\top, U (E_{ij} + E_{kl}) U^\top]$ given by
      \[
      \mu \int_0^1 (t + (1-t) D_i)^{-1} (t + (1-t) D_j)^{-1}
                       + (t + (1-t) D_k)^{-1} (t + (1-t) D_l)^{-1} dt.
      \]
      This is the sum of $T^2[U E_{ij} U^\top, U E_{ij} U^\top]$ with
      $T_D^2[U E_{kl} U^\top, U E_{kl} U^\top],$ meaning that
      $$
        T_D^2[U E_{ij} U^\top, U E_{kl} U^\top] = 0
      $$
      if $i \neq k$ or $j \neq l$.
      That is, $T_D^2[U - U^\top, U - U^\top]$ is diagonal with respect to
      the standard basis.  Put another way, the operator
      $C_{U^\top} T_D^2 C_U$ is diagonal where $C_U$ is the conjugation
      operator given by $C_U(Z) = UZU^\top.$  For every nonsingular $U$,
      the map $C_U$ preserves operator geometric
      means; that is,
      \begin{multline*}
          U A^{1/2} \left(A^{-1/2} B A^{-1/2}\right)^{1/2} A^{1/2} U^\top
        \\
        = (U A U^\top)^{1/2}
            \left((UAU^\top)^{-1/2} UBU^\top (UAU^\top)^{-1/2}\right)^{1/2}
          (U A U^\top)^{1/2}.
      \end{multline*}
      One can show this as follows:
      The geometric mean of $X\succ 0$ and $Y\succ 0$ is the unique positive
      definite $G$ such that $GX^{-1}G = Y$.  Taking $H = UGU^\top$, we have
      $$
        H U^{-\top}X^{-1}U^{-1} H = UGU^\top U^{-\top} X^{-1} U^{-1} UGU^\top
        = UGX^{-1}GU^\top = UYU^\top.
      $$
      Interestingly, Moln\'{a}r \cite{Molnar2009a} proved that {\it every}
      linear automorphism of the semidefinite cone over a complex Hilbert space
      that preserves geometric means is a conjugation operator.  The converse
      is also true, since the set of automorphisms of $\SSS_+^n$ is the set of
      conjugations given by the nonsingular $U$ (see a discussion in
      \cite{Tuncel2001} of G\"{u}ler's proof utilizing the proof technique of
      Waterhouse \cite{Waterhouse1989}), and as we observed above, it is easy
      to verify that the conjugation operator preserves operator geometric
      means.  Thus, we can make a natural comparison with the Nesterov--Todd
      scaling given by
      $$
        N[H,H] = \langle (V {D'}^{1/2} V^\top)^{-1} H
            (V {D'}^{1/2}V^\top)^{-1}, H \rangle,
      $$
      where $VD'V^\top$ is the spectral decomposition of operator geometric mean of
      $X$ and $\mu \tilde{X}$.  Notice that
      $$
        N[V E_{ij} V^\top, V E_{ij} V^\top] =
            \langle {D'}^{-1/2} E_{ij} {D'}^{-1/2}, E_{ij} \rangle
          = \frac{1}{\sqrt{D'_i D'_j}}
      $$
      and that $N$ is similarly diagonal with respect to that basis.
      (The conjugation $C_U N C_U^\top$ does not in general result in a diagonal
      matrix, so we need to take this different $V$ instead in order to
      diagonalise $N$.)  Notice that
      $$
        N[U E_{ij} U^\top, U E_{ij} U^\top]
        =e_i^\top U^\top G U e_i e_j^\top U^\top G U e_j
      $$
      whatever $U$ is.  Thus, when we form the matrix whose
      $ij$ entry is $N[U E_{ij} U^\top, U E_{ij} U^\top]$, we always obtain
      a rank-one matrix.  However, such a matrix formed from the integral scaling $T_D^2$ can have
      full rank.  We proceed with an example.
      Consider $D = (\epsilon, \epsilon^2, \ldots, \epsilon^n)$.  Notice that
      $\ln D_i - \ln D_j = (i-j) \ln \epsilon$, while
      $D_i - D_j = \epsilon^i - \epsilon^j.$  Thus, the $(i,j)$ entry of
      this matrix is on the order of $\epsilon^{-\min(i,j)}$.
      For sufficiently small $\epsilon$, then, the determinant of this matrix
      is dominated by the product along the diagonal, which is positive---it
      follows that this matrix is nonsingular.
    \end{examp}
  \end{section}

  \begin{section}{Local primal-dual metrics expressed as low rank updates}\label{sec:dfp}
    It may be impractical in many cases to evaluate the integral scaling
    from Section 2 exactly or to high enough accuracy.  Due to this, or perhaps
    due to numerical instabilities arising in computations, we may have to make
    do with an approximation $H$ that does not necessarily satisfy the
    equations $H s = x$ and $H \tilde{s} = \tilde{x}$.

    The second author \cite{Tuncel2001} constructed the following low-rank
    update which ``fixes'' such problems encountered by \emph{any} symmetric,
    positive-definite $H$:
    \begin{equation}\label{tuncelupdate}
      T_H^2 := H + a_1 xx^\top +  g_1 Hss^\top H
                + \tilde{a_1} \tilde{x} \tilde{x}^\top
      + \tilde{g_1} H \tilde{s} \tilde{s}^\top H
      + a_2 (x \tilde{x}^\top + \tilde{x} x^\top)
      + g_2 (H s \tilde{s}^\top H + H \tilde{s} s^\top H),
    \end{equation}
    where
    $$
      a_1 = \frac{\tilde{\mu}}{\vartheta (\mu \tilde{\mu} - 1)},
      \tilde{a}_1 = \frac{\mu}{\vartheta (\mu \tilde{\mu} - 1)},
      a_2 = \frac{-1}{\vartheta (\mu \tilde{\mu} - 1)},
    $$
    $$
      g_1 = \frac{\tilde{s}^\top H \tilde{s}}
                 {s^\top H s \tilde{s}^\top H \tilde{s} - (\tilde{s} H s)^2},
      \tilde{g}_1 = \frac{s^\top H s}
                 {s^\top H s \tilde{s}^\top H \tilde{s} - (\tilde{s} H s)^2},
      g_2 = \frac{s^\top H \tilde{s}}
                 {s^\top H s \tilde{s}^\top H \tilde{s} - (\tilde{s} H s)^2}.
    $$
    The second author \cite{Tuncel2001} proved that, as long as $H$ is
    positive-definite, $T_H^2$ is positive-definite, maps $s$ to $x$, and maps
    $\tilde{s}$ to $\tilde{x}$.

    As written, Equation \eqref{tuncelupdate} is somewhat unwieldy for purposes
    of analysis.  It is not immediately obvious that, in the case that the pair
    $(x,s)$ satisfy the centrality condition $x=\mu\tilde{x}$ (or,
    equivalently, $\mu \tilde{\mu}=1$), the formula collapses to a rank-two
    update.  (Indeed, it has a singularity there.)
    We prove that the above given complicated operator has an equivalent form
    as two, simple, consecutive updates due to the special structure of our
    set-up:
    \begin{thm}\label{dfpworks}
      Let $x,\tilde{x} \in \E$ and $s, \tilde{s} \in \E^*$ satisfy the
      following conditions (under the definitions $\delta_P := x - \mu
      \tilde{x}$ and $\delta_D := s - \mu \tilde{s}$):
      \begin{itemize}
        \item $0 < \iprod{\tilde{s}}{x} = \iprod{s}{\tilde{x}} =: \vartheta$
        \item $0 < \iprod{s}{x} =: \vartheta \mu$
        \item $0 < \iprod{\tilde{s}}{\tilde{x}} =: \vartheta \tilde{\mu}$
        \item $\iprod{\delta_D}{\delta_P} > 0.$
      \end{itemize}
      Further let $H$
      be some symmetric positive-definite matrix.
      Then, $H_2$ in the following formula is symmetric, positive-definite,
      maps $s$ to $x$, and maps $\tilde{s}$ to $\tilde{x}$:
      \begin{eqnarray}\label{dfpupdates}
        H_1 & := & H + \frac{1}{\iprod{s}{x}} x x^\top- \frac{1}{\iprod{s}{Hs}} H s s^\top H \\
        \nonumber
        H_2 & := & H_1 + \frac{1}{\iprod{\delta_D}{\delta_P}} \delta_P \delta_P^\top
                 - \frac{1}{\iprod{\delta_D}{H_1 \delta_D}} H_1 \delta_D\delta_D^\top H_1.
      \end{eqnarray}
    \end{thm}
    \begin{proof}
      Notice that $H_1 s = H s + x - Hs = x.$  Notice further that
      $\iprod{s}{\delta_P} = \iprod{\delta_D}{x} = 0$ by expanding
      the inner product conditions, so $H_2 s = H_1 s$.  Thus, $H_2$ maps $s$ to
      $x$.  Next, note that
      $H_2 \delta_D = H_1 \delta_D + \delta_P - H_1 \delta_D = \delta_P.$
      Thus, $H_2$ also maps $\delta_D$ to $\delta_P$.  Hence, $H_2$ maps $\tilde{s}$
      to $\tilde{x}$.

      We recall from the theory of quasi-Newton updates (see for instance Lemma
      9.2.1 in \cite{DS1983}) that the ``curvature condition'' $\iprod{s}{x} >
      0$ is necessary and sufficient to guarantee that $H_1$ is positive
      definite, and, given that $H_1$ is positive-definite, the curvature
      condition $\iprod{\delta_D}{\delta_P} > 0$ is necessary and sufficient
      for $H_2$ to be positive-definite.  Therefore, $H_2$ is positive-definite
      as well.
    \end{proof}

    Note that the positivity of the scalar products $\iprod{s}{x}$ and
    $\iprod{\delta_D}{\delta_P}$, together with the orthogonality conditions
    $\iprod{s}{\delta_P}=\iprod{\delta_D}{x}=0$ suffice for the above theorem
    to hold.  Thus, there may be a potential use of these formulae in
    classical quasi-Newton approaches. Such considerations are left for future
    work.

    We remark that we can apply the above formulas after switching $x$ and $s$
    and then inverting the resulting $T^2$ to obtain the following low-rank
    updates (under the same conditions):
    \begin{thm}\label{bfgsworks}
      Let $H$, $x$, $\tilde{x}$, $s$, and $\tilde{s}$ be as in Theorem
      \ref{dfpworks}.  Then $H_2$ in the following formula is symmetric,
      positive-definite, maps $s$ to $x$, and maps $\tilde{s}$ to $\tilde{x}$:
      \begin{equation}\label{bfgsupdates}
        H_1 := \left(I - \frac{x s^\top}{\iprod{s}{x}}\right) H
              \left(I - \frac{s x^\top}{\iprod{s}{x}}\right)
            + \frac{x x^\top}{\iprod{s}{x}}
      \end{equation}
      \begin{equation*}
        H_2 := \left(I - \frac{\delta_P \delta_D^\top}
                            {\iprod{\delta_D}{\delta_P}}\right) H_+
             \left(I - \frac{\delta_D \delta_P^\top}
                            {\iprod{\delta_D}{\delta_P}}\right)
           + \frac{\delta_P \delta_P^\top}{\iprod{\delta_D}{\delta_P}}.
      \end{equation*}
    \end{thm}
    \begin{proof}
      We again see that $H_1 s = x$ and $H_2 \delta_D = \delta_P$ by
      correctness of the BFGS/DFP update.  We compute
      $$
        H_2 s =
             \left(I - \frac{\delta_P \delta_D^\top}
                            {\iprod{\delta_D}{\delta_P}}\right) H_1
             \left(I - \frac{\delta_D \delta_P^\top}
                            {\iprod{\delta_D}{\delta_P}}\right) s
           + \frac{\delta_P \delta_P^\top}{\iprod{\delta_D}{\delta_P}} s
             =
             \left(I - \frac{\delta_P \delta_D^\top}
                            {\iprod{\delta_D}{\delta_P}}\right) x
           + 0 = x.
      $$
      $H_1$ is positive-definite because the curvature condition $\iprod{s}{x}
      > 0$ is satisfied.  $H_2$ is positive-definite because the curvature
      condition $\iprod{\delta_D}{\delta_P} > 0$ is satisfied.
    \end{proof}

    Due to the variational interpretations of quasi-Newton update formulae, we
    have the corresponding interpretations in our set-up (i.e., $H_1$ is the
    closest---minimum distance---self-adjoint operator to $H$ satisfying $H_1 s
    = x$; similarly for $H_2$ and $H_1$).
    Moreover, as in \cite{Tuncel2001}, the above formulae can be used in convex
    combinations (analogous to \emph{Broyden's convex class} in the classical
    quasi-Newton context) due to convexity of $\cT_1$.  Of course, we may also
    use the formula for $H_1$ in Theorem \ref{dfpworks} together with the
    formula for $H_2$ in Theorem \ref{bfgsworks}, etc.  In the next section, we
    will look at these properties more deeply from a strict primal-dual
    symmetry viewpoint.
  \end{section}

  \begin{section}{Primal-dual symmetry based on the local metrics}\label{sec:pd-symmetry}
    In \cite{Tuncel2001} a very general framework for primal-dual symmetric
    algorithms were provided.  In this section, we make the richness for the
    choice of such algorithms within the framework provided by the sets $\cT_0,
    \cT_1, \cT_2$, more explicit.  Basically, \emph{every} consistent choice of
    a local primal-dual metric $T^2$ from any of the sets $\cT_0, \cT_1, \cT_2$
    can be used to design a primal-dual symmetric interior-point algorithm as
    we prove below.

\begin{prop}
\label{prop:6.1}
Let $(x,s) \in \inte(K) \oplus \inte(K^*)$.  Then, for every pair $H, T \in \cT_0^2 (x,s)$,
\[
\frac{1}{2}(H+T), \left(\frac{H^{-1}+T^{-1}}{2}\right)^{-1} \in \cT_0^2 (x,s).
\]
The same property holds for $\cT_1^2(x,s)$ and for $\cT_2^2(\eta;x,s)$ (for every $\eta \geq 1$).
\end{prop}

\begin{proof}
Follows from the definitions and convexity of $\cT_0^2, \cT_1^2, \cT_2^2$.
\end{proof}

\begin{cor}
For every convex cone $K \subset \E$ and for every pair $(x,s) \in \inte(K) \oplus \inte(K^*)$,
the sets $\cT^2_0(x,s)$, $\cT^2_1(x,s)$ and $\cT^2_2(\eta;x,s)$ (for every $\eta \geq 1$) are geodesically convex.
\end{cor}

\begin{proof}
Since $\cT_2^2(\eta;x,x)$ is a closed subset of $\SSS^n_{++}$, employing Proposition \ref{prop:6.1} above
and Lemma 2.3 of \cite{LimSIOPT2011}, we conclude that $\cT_2^2(\eta;x,x)$ is geodesically convex
for all $\eta$.  For $\cT_0$ and $\cT_1$ we can adapt the proof of the same lemma (even though
our sets are not closed, we can argue that all limits of all mean iterations on the elements of
$\cT_0^2$ and $\cT_1^2$ are positive-definite
and hence stay in the corresponding set).
\end{proof}

The above corollary indicates a way to convert \emph{any} consistent choice of $T^2$ to a scaling for
a primal-dual symmetric interior-point algorithm in the sense of \cite{Tuncel1998}.
(Simply take the operator geometric mean of the consistent choice of $T^2$
with the inverse of the same formula for $T^2$ applied to $(P)$ and $(D)$ switched.)

  \end{section}

  \begin{section}{Primal-dual Hessian local metrics based on a single scaling point}\label{sec:scaling-point}
    In this section, we ask and partially answer the following question:
    For which $\vartheta$-LHSCBs $F_*$ does there exist a unique scaling point $w
    \in \inte(K^*)$ such that $F_*''(w) \in \cT_1^2 (x,s)$ for every $x \in \inte(K)$ and $s \in \inte(K^*)$?

    Note that, on the one hand, for every $\vartheta$-LHSCB $F_*$,
    there exists a scaling point $w \in \inte(K^*)$ such that $F_*''(w) \in \cT_0^2(x,s)$
    for every $x \in \inte(K)$ and $s \in \inte(K^*)$ (as it was already proved by Nesterov and Todd
    \cite{NT1997}; see \cite{Tuncel2001}, Theorem 3.1).  On the other hand, the
    Nesterov--Todd scaling point given by the geodesic mean of $s$ and
    $-F'(x)$ provides an example of such a $w$ in the symmetric cone case (in
    this special case, $F_*''(w) \in \cT_1^2 (x,s)$).
    We show that this property does not generalise to slices of a symmetric cone that are not themselves
    symmetric cones.

    \begin{lem}
      Let $L$ be a linear subspace of $\SSS^n$ that contains $I$ but is
      not closed under matrix squaring.  Then there exists a $B \in L$ such
      that $\tr B = 0$ and $B^2 \not\in L$.
    \end{lem}
    \begin{proof}
      Select a $C$ such that $C^2 \not\in L$.  Let $t := \tr C/n$ and
      $B := C - tI$.  Then,
      $\tr B = 0$ and $B^2 = C^2 - 2tC + t^2 I$.  The latter two terms of
      this expansion lie in $L$ while $C^2$ does not, so $B^2 \not\in L$.
    \end{proof}

    \begin{prop}
    \label{prop:6.2}
      Let $L$ be a linear subspace of $\SSS^n$ that contains $I$ but is not
      closed under matrix squaring.  Choose the barrier $F(X) = -\ln \det X$
      for the cone $K := \SSS^n_+ \cap L$.  (The dual cone of $K$ is
      $K^* = \{S \in L : \forall X \in K, \tr SX \geq 0\}.$)
      There exist $X \in \inte(K)$ and $S \in \inte(K^*)$ such that, for all $W \in \inte(K)$,
      either $F''(W)[X] \neq S$ or $F''(W)[-F'_*(S)] \neq -F'(X)$.
    \end{prop}
    \begin{proof}
      Assume for the purpose of contradiction that there are $n$ and $L$ such
      that, for every choice of $X \in K$ and $S \in K^*$, there exists a $W$
      for which $F''(W)[X] = S$ and $F''(W)[-F'_*(S)] = -F'(X)$.

      For $Z \in \inte(K)$, we compute
      $$
        F'(Z) = -\Pi_L(Z^{-1})
      $$
      and
      $$
        F''(Z)[H] = \lim_{h \to 0^+} -\frac{1}{h} \Pi_L((Z+hH)^{-1} - Z^{-1})
        = \Pi_L(Z^{-1}HZ^{-1}),
      $$
      where $\Pi_L$ is the orthogonal projection onto $L$.

      Select $B \in L$ such that $\tr B = 0$ and $B^2 \not\in L$.
      Let $S := X := I + \epsilon B$; we shall choose $\epsilon > 0$ later.
      (For sufficiently small $\epsilon$, $X$ lies in the interior of
      $K$ and $S$ in the interior of $K^*$.)
      Notice that this choice of $S$ and $X$ implies that
      $S = F''(W)[X] = \Pi_L(W^{-1} X W^{-1}).$

      Next, we check that $W = I$ is the only solution to $S = F''(W)[X]$.
      Suppose $W$ is such that $S = F''(W)[X]$.
      Consider the problem
      $$
        \begin{array}{rrcl}
          \min & \tr(Z^{-1} X)\\
          \textrm{subject to} & \tr(ZX) & \leq & \tr(WX)\\
               & Z & \in & L\\
               & Z & \succeq & 0.\\
        \end{array}
      $$
      Notice that the objective is strictly convex (this is indeed related
      to the fact that the long-step Hessian
      estimation property holds for $-\ln \det(\cdot)$) and that the feasible region
      is nonempty, compact and convex
      (since $X$ is positive definite).  Thus this problem has a unique optimal
      solution.  This optimal solution does not lie in the boundary of the
      positive semidefinite cone since the objective is $+\infty$ there and
      finite elsewhere.
      $Z := W / 2$ is a feasible solution satisfying the positive
      semidefiniteness constraint and the linear inequality strictly, so
      Slater's condition holds.  Notice that the gradient of the objective is $-Z^{-1} X Z^{-1}$.

      By the Karush-Kuhn-Tucker theorem, every point $Z$ for which $\tr(ZX) =
      \tr(WX)$ and there exists a $\lambda \geq 0$ and $\ell^\perp \in L^\perp$
      such that $Z^{-1} X Z^{-1} - \lambda X + \ell^\perp = 0$ is optimal.
      Note that $Z = W$ is a solution with $\lambda = 1$.  Note that $Z :=
      \tr(WX) I / \tr X$ is also a solution with $\lambda =
      \left[\tr(X)/\tr(WX)\right]^2.$  Thus $W$ must be a scalar multiple of $I$; since
      $\Pi_L(W^{-1} X W^{-1}) = X$, that scalar must be 1.

      Define $\tilde{X} := -F'_*(S)$.  Then, we must have
      $F''(W)[\tilde{X}] = -F'(X)$; $\tilde{X} =
      \Pi_L(X^{-1}) = X^{-1} - P$ for some $P \in L^\perp$.
      Observe that, for sufficiently small $\epsilon$,
      $$
        X^{-1} = I - \epsilon B + \epsilon^2 B^2 + O(\epsilon^3).
      $$
      Both $I$ and $B$ lie in $L$, so
      $$
        P = \Pi_{L^\perp}(\epsilon^2 B^2) + O(\epsilon^3).
      $$

      We compute, for sufficiently small $\epsilon$,
      \begin{eqnarray*}
      n & = & \tr S = \tr \left[-F'(\tilde{X})\right] = \tr \left[(X^{-1} - P)^{-1}\right]
        = \tr \left[X + XPX + XPXPX + O(\epsilon^6)\right]\\
        & = & \tr \left[I + \epsilon B + P + \epsilon BP + \epsilon PB
                 + \epsilon^2 BPB + P^2 + O(\epsilon^5)\right]\\
        & = & n + 0 + 0 + 0 + 0 + \epsilon^2 \tr(B^2 P) + \tr(P^2) + O(\epsilon^5).
      \end{eqnarray*}
      Noting that $P = \Pi_{L^\perp}(\epsilon^2 B^2)$, we see that
      $\epsilon^2 \tr(B^2 P) = \tr(P^2)$, which is the squared Frobenius norm
      of $P$.  By our choice of $B$, this is positive.  Thus, for sufficiently
      small $\epsilon$, $\tr X = \tr S = \tr \left[-F'(\tilde{X})\right] > \tr X,$
      a contradiction.
    \end{proof}

    There is a five-dimensional semidefinite cone slice where a $\cT_1$ scaling
    defined by a single point does not exist for all $x$ and $s$, namely the
    cone of symmetric, positive semidefinite matrices with the sparsity pattern
    $$
      \left(\begin{array}{ccc}
        * & * & *\\
        * & * & 0\\
        * & 0 & *\\
      \end{array}\right).
    $$

    This cone (described above as a slice of $3$-by-$3$ positive semidefinite
    cone) is also known as the \emph{Vinberg cone}.  (Here we are working with
    its SDP representation).  It is the smallest-dimensional homogeneous cone
    that is not self-dual, whence not a symmetric cone.

    Let us note that if $L$ is a linear subspace of $\SSS^n$ that contains $I$
    and is closed under matrix squaring, then
    \[
    \forall U,V \in L, \,\,\,\,
    (U+V)^2 -U^2 -V^2 = UV +VU \in L.
    \]
    Moreover, since for every pair $U,V \in L$, the equation
    \[
    U\left(U^2V+VU^2\right) + \left(U^2V+VU^2\right)U
    = U^2\left(UV+VU\right) + \left(UV+VU\right)U^2
    \]
    is self-evident, we have an Euclidean Jordan algebra over $L$.  Hence,
    $\left(\SSS_{+}^n \cap L\right)$ is an SDP representation of a symmetric
    cone and indeed the function $-\ln \det(\cdot): \inte(K) \to \R$ is a
    self-scaled barrier for $\left(\SSS_{+}^n \cap L\right)$.  Therefore,
    Proposition~\ref{prop:6.2} proves that among all SDP-representable cones
    (as a slice of $\SSS^n_+$), symmetric cones are the only ones for which
    \[
    \left\{ F''_*(w): w \in \inte(K^*) \right\} \cap \cT_1^2 (x,s) \neq \emptyset,
    \,\,\,\,
    \forall (x,s) \in \inte(K) \oplus \inte(K^*),
    \]
    where $F(X):= -\ln\det(X)$ (over the representing cone $\SSS_+^n$).

    The proof technique of Proposition~\ref{prop:6.2} is more generally
    applicable.  As a result, a more general and sharper characterization of the underlying
    behaviour is possible.  This will be addressed, in detail, elsewhere.
  \end{section}

  \begin{section}{The norm of the low-rank updates near the central path}\label{sec:analysis}
    We know that if we can compute $T \in \cT_2^2(\eta;x,s)$ efficiently, with
    ensuring $\xi^*=O(1)$ and $\eta = O(1)$ for all iterates, then we will have
    one of the most important ingredients of a
    general primal-dual interior-point algorithms with iteration complexity
    $O\left(\sqrt{\vartheta}\ln(1/\epsilon)\right)$.

    Up to this point, we have seen many ways of constructing $T \in
    \cT_1(x,s)$.  However, we also discovered that we cannot expect to have a
    $w \in \inte(K^*)$ such that $F''_*(w) \in \cT_1^2(x,s)$, in general.  We
    will later present and analyse an algorithm for convex programming based on
    Mizuno, Todd, and Ye's predictor-corrector approach \cite{MTY1993}.

    We will first assume explicit access to oracles computing a
    $\vartheta$-self-concordant primal barrier $F$ for the primal cone and the
    conjugate barrier $F_*$ for the dual cone.  We will argue later that the
    algorithms can be modified to work without an explicit $F_*$ oracle.

    These oracles may be expensive; it may be unreasonable (or simply
    unnecessary) to try to compute the dual integral scaling directly to high
    precision at every iteration.  We thus consider computing an approximation
    $H$ to $T_D^2$ and then using a low-rank update to $H$ to get a scaling in
    $T_H^2 \in \cT_1(x, s)$.  Specifically, we approximate the operator
    integral by evaluating it at the midpoint of the line segment joining the
    two extremes of the positive-definite operator $\left\{F_*''(s -t\delta_D):
    t \in [0,1] \right\}$.  Then, we use the low-rank updates of Section
    \ref{sec:dfp} to restore membership in $\cT_1^2(x,s)$.

    Define
    \[
      \check{s} := \frac{s + \mu \tilde{s}}{2}
    \mbox{ and }
      H := \mu F''_*(\check{s}),
    \]
    and take $H_1$ and $T_H^2 := H_2$ as in \eqref{dfpupdates}.

    Our analysis hinges on the resulting scaling (local metric $T_H^2$) being
    close to $\mu F''_*(s)$ and $\left[\mu F''(x)\right]^{-1}$ (in the sense of Definition
    \ref{defi:2.1}) in every iteration of the algorithm.  We therefore devote
    the remainder of this section to computing bounds, somehow dependent on the
    error in approximating $T_D^2$ by $H$, of the additional error introduced
    by the low-rank updates.  We will prove in this section that for every pair
    of interior points $(x,s) \in \inte(K) \oplus \inte(K^*)$ for which $x$ is
    close to $\mu \tilde{s}$, and hence $s$ is close to $\mu \tilde{x}$ (in the
    sense that e.g., $\norm{\delta_D}_s < 1/50$), $\xi^*$ is $O(1)$ (in fact,
    less than 4/3), and $T_H^2$ is a $4/3$-approximate solution to the SDP
    defining $\cT(1;x,s)$.  We made a particular choice of $1/50$ for the
    neighbourhood parameter; indeed, a continuous parametrization of the
    following analysis with respect to the neighbourhood parameter is possible
    (and implicit).

      For every $s \in \inte(K^*)$, we define the operator norm
      \[
        ||M||_s = \sup_{||u||_s \leq 1} ||M u||_s^*.
      \]
      Observe that
      \[
        2 \left(h h^\top - u u^\top \right) = (h-u)(h+u)^\top + (h+u)(h-u)^\top.
      \]
      We make considerable use of the following difference-of-squares bound:
      \begin{lem}
      \label{lem:dif-of-sqr}
        Let $h$ and $u$ lie in $\E$ and $s \in \inte(K^*)$.
        Then
        $$
          ||h h^\top - u u^\top||_s \leq ||h-u||_s^* ||h+u||_s^*.
        $$
      \end{lem}
      \begin{proof}
        By the triangle inequality,
        \begin{eqnarray*}
               2 ||h h^\top - u u^\top||_s
             & = & \norm{(h-u)(h+u)^\top + (h+u)(h-u)^\top}_s
          \\
          & \leq & \norm{(h-u)(h+u)^\top}_s + \norm{(h+u)(h-u)^\top}_s.
        \end{eqnarray*}
        Now we compute
        \begin{eqnarray*}
            \norm{(h-u)(h+u)^\top}_s
          & = & \sup_{||z||_s \leq 1} \norm{(h-u) \iprod{z}{h+u}}_s^*
          \\
          & = & \sup_{||z||_s \leq 1} \iprod{z}{h+u} \norm{h-u}_s^*
          = \norm{h+u}_s^* \norm{h-u}_s^*;
        \end{eqnarray*}
        and similarly
        $$
            \norm{(h+u)(h-u)^\top}_s
          = \norm{h+u}_s^* \norm{h-u}_s^*.
        $$
        Adding these together gives the advertised result.
      \end{proof}

      The next lemma is used many times in the following analysis.
      \begin{lem}\label{lem:7.2}
        Let $s \in \inte(K^*)$, $h \in \E^*$ such that $\norm{h}_s < 1.$
        Then,
        \[
               \sup_{u \in \E^*: \norm{u}_s \leq 1} \norm{F_*''(s+h) u}_s^*
          \leq \frac{1}{\left(1-\norm{h}_s\right)^2}.
        \]
      \end{lem}
      \begin{proof}
         Let $s$ and $h$ be as in the statement of the lemma.  Then
         $F_*''(s+h) \preceq \frac{1}{\left(1-\norm{h}_s\right)^2} F_*''(s)$ by
         the Dikin ellipsoid bound.  Thus, the maximum eigenvalue of
         $\left[F_*''(s)\right]^{-1/2} F_*''(s+h) \left[F_*''(s)\right]^{-1/2}$
         is bounded above by $\frac{1}{\left(1-\norm{h}_s\right)^2}$.  The
         square of this quantity is an upper bound on the largest eigenvalue of
         $\left[F_*''(s)\right]^{-1/2} F_*''(s+h)
         \left[F_*''(s)\right]^{-1}F_*''(s+h)\left[F_*''(s)\right]^{-1/2}$.
         Therefore, the supremum in the statement of the lemma is bounded above
         by $\frac{1}{\left(1-\norm{h}_s\right)^2}$ as desired.
      \end{proof}

For the rest of the analysis, we will increase the use of explicit absolute constants,
for the sake of concreteness.  We either write these constants as ratios of two integers
or as decimals which represent rationals with denominator: $10^6$ (whence, we are able to express
the latter constants \emph{exactly} as decimals with six digits after the point).

    \begin{subsection}{Zeroth-order low-rank update}
      We bound the operator norm of the zeroth-order low-rank update in a small
      neighbourhood of the central path defined by the condition
      $||\delta_D||_s \leq \frac{1}{50}$.
\begin{thm}
\label{thm:7.3}
Assume $\norm{\delta_D}_s \leq 1/50$.  Then
\begin{enumerate}
  \item $||F''_*(s) \delta_D||_s^* = ||\delta_D||_s;$
  \item $\norm{F''_*(\bar{s}) \delta_D}_s^* \leq 1.041233 \norm{\delta_D}_s;$
  \item for every $v \in \E^*$,
    $$
      \abs{\iprod{v}{\mu F''_*(\check{s}) \delta_D - \delta_P}}
      \leq 0.265621 \mu \norm{v}_s \norm{\delta_D}_s^2;
    $$
  \item $\norm{\mu F''_*(\check{s}) \delta_D - \delta_P}_s^*
      \leq 0.265621 \mu \norm{\delta_D}_s^2.$
\end{enumerate}
\end{thm}
\begin{proof}
\begin{enumerate}
\item This follows straightforwardly from the definitions.
\item Recall that
  $F''_*(\bar{s}) \preceq \frac{1}{(1-||\delta_D||)^2} F''_*(s)$.
  Now we apply the previous part.  Next, we notice that
  $\frac{1}{(1-||\delta_D||_s)^2} \leq 1.041233.$
\item Let $f(t) = \iprod{u}{F'_*(\check{s} + t \delta_D)}.$
  We consider an order-two Taylor expansion of $f$ around zero; we see that, for
  every $t \in [-1/2, 1/2]$, there exists a $\bar{t} \in [\min(0,t), \max(0,t)]$ such that
  $$
    f(t) = f(0) + t f'(0) + \frac12 t^2 f''(\bar{t}).
  $$
  Notice that
  $$
    f''(\bar{t}) = F'''_*(\check{s} + \bar{t} \delta_D)
                         [u, \delta_D, \delta_D]
  \leq 2 ||u|| ||\delta_D||^2,
  $$
  where both norms are the local $\left(\check{s}+\bar{t}\delta_D\right)$ norms
  (we used self-concordance property of $F_*$).
  We then apply the Dikin ellipsoid bound to these norms to relate
  $$
         ||u||_{\check{s} + \bar{t} \delta_D}
    \leq \frac{1}{1 - ||\delta_D||_s} ||u||_s
  $$
  and similarly for $||\delta_D||$.  Consequently, using
  $1/(1-||\delta_D||_s) \leq 1.020409$, we see that
  $$
    |f''(\bar{t})| \leq 125000/117649 ||u||_s ||\delta_D||_s^2.
  $$
  Thus, for some $\bar{t}_1 \in [-1/2, 0]$ and
  $\bar{t}_2 \in [0, 1/2]$, we have
  $$
    f(1/2) - f(-1/2) = f'(0)
        + \frac18 (f''(\bar{t}_1) - f''(\bar{t}_2)).
  $$
  Consequently,
  $$
       |f'(0) - f(1/2) + f(-1/2)|
  \leq 31250/117649 ||u||_s ||\delta_D||_s^2.
  $$
  Notice that, by substitution and the chain rule,
  \begin{itemize}
    \item $f'(0) = \iprod{u}{F''_*(\check{s}) \delta_D};$
    \item $f(-1/2) = \iprod{u}{F_*'(\mu \tilde{s})}
                    = -\iprod{u}{x/\mu};$
    \item $f(1/2) = \iprod{u}{F_*'(s)}
                   = -\iprod{u}{\tilde{x}}.$
  \end{itemize}
  The claimed bound now follows.
\item We use the definition of a dual norm:
  \begin{eqnarray*}
       ||F''_*(\check{s}) \delta_D - x/\mu + \tilde{x}||_s^*
     & = & \sup_{||u||_s = 1}
       \iprod{u}{F''_*(\check{s}) \delta_D - x/\mu + \tilde{x}}
  \\
  & \leq & \sup_{||u||_s = 1} 0.265621 ||u||_s ||\delta_D||_s^2
     = 0.265621 ||\delta_D||_s^2.
  \end{eqnarray*}
\end{enumerate}
\end{proof}

\begin{lem}
\label{lem:7.4}
Assume $\norm{\delta_D}_s \leq 1/50$.  Then,
 the zeroth-order low-rank update has small norm:
\begin{enumerate}
  \item for every $v \in \E^*$,
  $$
    \abs{\iprod{v}{F''_*(\check{s})[\delta_D]}} \leq 1.020305 \norm{v}_s \norm{\delta_D}_s;
  $$
  \item for every $v \in \E^*$ and $\bar{s} \in [s, \mu \tilde{s}]$,
  $$
    \abs{F'''_*(\bar{s})[v, \delta_D, \mu \tilde{s}]} \leq 2.124965 \norm{v}_s \norm{\delta_D}_s \sqrt{\vartheta};
  $$
  \item $\norm{Hs-x}_s^* \leq 2.082788 \sqrt{\vartheta} \mu \norm{\delta_D}_s;$
  \item $\norm{x}_s^* \leq 1.020409 \sqrt{\vartheta} \mu;$
  \item $\norm{Hs}_s^* \leq 1.020305 \sqrt{\vartheta} \mu;$
  \item $\norm{Hs+x}_s^* \leq 2.040714 \sqrt{\vartheta} \mu;$
  \item $\iprod{s}{Hs} \geq 0.980100 \vartheta \mu.$
\end{enumerate}
\end{lem}
\begin{proof}
\begin{enumerate}
  \item We compute, using Cauchy-Schwarz and the Dikin ellipsoid bound,
  $$
      \abs{\iprod{v}{F''_*(\check{s})[\delta_D]}}
    = \abs{F''_*(\check{s})[\delta_D, v]}
  \leq \norm{\delta_D}_{\check{s}} \norm{v}_{\check{s}}
  \leq 1.020305 \norm{\delta_D}_s \norm{v}_s.
  $$
  \item We compute
  $$
    \abs{F'''_*(\bar{s})[v, \delta_D, \mu \tilde{s}]}
  \leq 2 \norm{v}_{\bar{s}} \norm{\delta_D}_{\bar{s}}
        \norm{\mu \tilde{s}}_{\bar{s}}
  \leq 2.124965 \norm{v}_{s} \norm{\delta_D}_{s}
        \norm{\mu \tilde{s}}_{\mu \tilde{s}}.
  $$
  \item We write
$$
  Hs = \mu F''_*(\check{s}) (\mu \tilde{s})
     + \mu F''_*(\check{s}) \delta_D.
$$
On the first term, we perform a Taylor expansion around
$\mu \tilde{s}$; for every $v$ there is a $\bar{s}$ on the line
segment between $s$ and $\mu \tilde{s}$ such that
\begin{eqnarray*}
  F''_*(\check{s})[\mu \tilde{s}, v]
      & = & F''_*(s-\delta_D) [\mu \tilde{s}, v]
      + \frac12 F'''_*(\bar{s})[\mu \tilde{s}, \delta_D, v]
      \\
      & = & \iprod{v}{x}/\mu
      + \frac12 F'''_*(\bar{s})[\mu \tilde{s}, \delta_D, v].
      \\
      & \leq & \iprod{v}{x}/\mu
      + 424993/400000 \norm{v}_s \norm{\delta_D}_s \sqrt{\vartheta}.
\end{eqnarray*}
We also bound (using the Dikin ellipsoid bound first, followed by Lemma \ref{lem:7.2})
$$
       \iprod{v}{F''_*(\check{s}) \delta_D}
  \leq 1.020305 ||v||_s ||\delta_D||_s.
$$
Adding these bounds and taking a supremum over all $v$ such that $||v||_s = 1$,
since $\vartheta \geq 1$, yields the bound
$$
       ||Hs-x||_s^*
  \leq 2.082788 \sqrt{\vartheta} \mu ||\delta_D||_s,
$$
as desired.
  \item This follows directly from the Dikin ellipsoid bound.
  \item Notice that
    $$
           ||Hs||_s^*
         = \mu \iprod{s}{F''_*(\check{s}) (F''_*(s))^{-1}
                 F''_*(\check{s}) s}^{1/2}
      \leq \mu \frac{\iprod{s}{F''_*(s) s}^{1/2}}
                          {(1 - ||\delta_D||/2)^2}
      \leq 1.020305 \sqrt{\vartheta} \mu.
    $$
  \item We apply the triangle inequality to the last two parts.
  \item Note that $\iprod{s}{Hs}
    = \mu F''(\check{s})[s,s]
    \geq \mu (1 - \delta_D/2)^2 F''(s)[s,s]
    \geq 0.9801 \vartheta \mu.$
\end{enumerate}
\end{proof}

\begin{thm}\label{thm:dfp1}
Assume $\norm{\delta_D}_s \leq 1/50$.  Then
$$
  \norm{  \frac{x x^\top}{\iprod{s}{x}}
         - \frac{H s s^\top H}{\iprod{s}{Hs}}}_s
    \leq 6.462628 \mu \norm{\delta_D}_s.
$$
\end{thm}
\begin{proof}
  We write the first low-rank update as
  $$
      \frac{x x^\top - H s s^\top H}{\iprod{s}{x}}
    + \left(\frac{1}{\iprod{s}{x}}
    - \frac{1}{\iprod{s}{Hs}}\right) H s s^\top H.
  $$
  Then, using the triangle inequality, Lemma~\ref{lem:dif-of-sqr}, and the bound $||v v^\top||_s \leq
  ||v||_s^{*2}$ we bound its norm above by
  $$
      \frac{1}{\iprod{s}{x}} \norm{x-Hs}_s^* \norm{x+Hs}_s^*
    + \left|\frac{1}{\iprod{s}{x}} - \frac{1}{\iprod{s}{Hs}}\right|
        \norm{Hs}_s^{*2}.
  $$
  The first term is bounded above by
  $$
    531296828829/125000000000 \mu ||\delta_D||_s.
  $$
  To bound the second term, note that
  $$
         \left|\frac{1}{\iprod{s}{x}} - \frac{1}{\iprod{s}{Hs}}\right|
       = \left|\frac{\iprod{s}{Hs-x}}{\vartheta \mu \iprod{s}{Hs}}\right|
    \leq \frac{||s||_s ||Hs-x||_s^*}{\vartheta \mu \iprod{s}{Hs}}
    \leq 2.082788 \frac{||\delta_D||_s}{\iprod{s}{Hs}}.
  $$
  Now, we bound $\iprod{s}{Hs}$ below by $0.980100 \vartheta \mu$
  to get a bound of
  $$
    520697/245025 \frac{||\delta_D||_s}{\vartheta \mu}.
  $$
  The bound $||Hs||_s^* \leq 1.020305 \sqrt{\vartheta} \mu$ then
  gives an overall bound on the second term of
  $$
    179192457821897/81000000000000 \mu ||\delta_D||_s
  $$
  Adding fractions gives (something slightly stronger than) the desired bound.
\end{proof}

\begin{lem}
\label{lem:7.6}
Assume $\norm{\delta_D}_s \leq 1/50$.
Let $H_1 := H + \frac{xx^\top}{\iprod{s}{x}}
- \frac{Hss^\top H}{\iprod{s}{Hs}}$.  Then,
\begin{enumerate}
  \item $\norm{H \delta_D - \delta_P}_s^* \leq 0.265621 \mu \norm{\delta_D}_s^2;$
  \item $\iprod{\delta_D}{H \delta_D} \geq
    0.960400 \mu \norm{\delta_D}_s^2$;
  \item $\iprod{\delta_D}{\delta_P} \geq
    0.955087 \mu \norm{\delta_D}_s^2$;
  \item $\norm{H\delta_D}_s^* \leq
    1.020305 \mu \norm{\delta_D}_s;$
  \item $\norm{\delta_P}_s^* \leq
    1.025618 \mu \norm{\delta_D}_s;$
  \item $\norm{H\delta_D + \delta_P}_s^* \leq
    2.045923 \mu \norm{\delta_D}_s;$
  \item $\norm{H_1 \delta_D - H \delta_D}_s^* \leq
    0.276518 \mu \norm{\delta_D}_s^2;$
  \item $||H_1 \delta_D - \delta_P||_s^* \leq
    0.542139 \mu ||\delta_D||_s^2;$
  \item $||H_1 \delta_D||_s^* \leq
    1.025836 \mu ||\delta_D||_s;$
  \item $||H_1 \delta_D + \delta_P||_s^* \leq
    2.051454 \mu ||\delta_D||_s;$
  \item $\iprod{\delta_D}{H_1 \delta_D} \geq
    0.944244 \mu ||\delta_D||_s^2.$
\end{enumerate}
\end{lem}
\begin{proof}
\begin{enumerate}
  \item This was proven in Theorem~\ref{thm:7.3}, part (4).
  \item This follows from the Dikin ellipsoid bound;
         $H \succeq 0.960400 \mu F''_*(s)$.
  \item Notice that
    $$
        \iprod{\delta_D}{\delta_P}
      =   \iprod{\delta_D}{H \delta_D}
        + \iprod{\delta_D}{\delta_P - H \delta_D}.
    $$
    We bound the second term by Cauchy-Schwarz:
    $$
           \iprod{\delta_D}{\delta_P - H \delta_D}
      \leq ||\delta_D||_s ||H \delta_D - \delta_P||_s^*
      \leq 0.265621 \mu ||\delta_D||_s^3.
    $$
    Using this with the bound from the previous part gives the
    advertised inequality.
  \item We compute, using Lemma~\ref{lem:7.2}
    \begin{eqnarray*}
         ||H\delta_D||_s^*
       = \mu \iprod{\delta_D}{F''_*(\check{s})
           \left(F''_*(s)\right)^{-1} F''_*(\check{s}) \delta_D}^{1/2}
    & \leq & 1.020305 \mu \iprod{\delta_D}{F''_*(s) \delta_D}^{1/2}
       \\
       & = & 1.020305 \mu ||\delta_D||_s,
    \end{eqnarray*}
    as advertised.
  \item We use the triangle inequality followed by parts (1) and (4):
  $$
    \norm{\delta_P}_s^* \leq
    \norm{H \delta_D}_s^* + \norm{\delta_P - H \delta_D}_s^*
    \leq 1.020305 \mu \norm{\delta_D}_s + 0.265621 \mu \norm{\delta_D}_s^2
    \leq 1.025618 \mu \norm{\delta_D}_s.
  $$
  \item We use the triangle inequality, part (4) and the bound $||\delta_D||_s \leq 1/50$:
    \begin{eqnarray*}
           ||H \delta_D + \delta_P||_s^*
      \leq 2 ||H \delta_D||_s^* + ||H \delta_D - \delta_P||_s^*
      & \leq & 2.040610 \mu ||\delta_D||_s
         + 0.265621 \mu ||\delta_D||_s^2
      \\
      & \leq & 2.040610 \mu ||\delta_D||_s
         + 265621/50000000 \mu ||\delta_D||_s,
    \end{eqnarray*}
    which is the claimed bound.
  \item Recall that $\iprod{\delta_D}{x} = 0$, so
    $$
      H \delta_D - H_1 \delta_D = \frac{\iprod{s}{H \delta_D}}
        {\iprod{s}{Hs}}Hs.
    $$
    Now, we bound using $\iprod{s}{\delta_P}=0$, triangle inequality and part (1):
    $$
           |\iprod{s}{H\delta_D}|
         = |\iprod{s}{\delta_P} + \iprod{s}{H\delta_D-\delta_P}|
      \leq 0 + ||s||_s ||H\delta_D-\delta_P||_s^*
      \leq 0.265621 \sqrt{\vartheta} \mu ||\delta_D||_s^2
    $$
    and recall (Lemma~\ref{lem:7.4} part (7))
    $$
      \iprod{s}{Hs} \geq 0.980100 \vartheta \mu
    $$
    and (Lemma~\ref{lem:7.4} part (5))
    $$
      ||Hs||_s^* \leq 1.020305 \sqrt{\vartheta} \mu.
    $$
    Thus,
    $$
           ||H_1 \delta_D - H \delta_D||_s^*
      \leq 0.276518 \mu ||\delta_D||_s^2.
    $$
  \item We use the triangle inequality followed by parts (1) and (7).
  \item We use the triangle inequality followed by parts
 (4), (7) and the fact that\\
 $\norm{\delta_D}_s \leq 0.020000$.
  \item We use the triangle inequality and parts (5) and (9).
  \item
    We compute, using previous parts of this lemma,
    \begin{eqnarray*}
               \iprod{\delta_D}{H_1 \delta_D}
         & = & \iprod{\delta_D}{\delta_P} + \iprod{\delta_D}{H_1 \delta_D - \delta_P}\\
      & \geq & 0.955087 \mu \norm{\delta_D}_s^2
           - \norm{\delta_D}_s \norm{H_1 \delta_D - \delta_P}_s^* \\
      & \geq & 0.955087 \mu ||\delta_D||_s^2
           - 0.542139 \mu ||\delta_D||_s^3\\
      & \geq & 0.944244 \mu ||\delta_D||_s^2.
    \end{eqnarray*}
\end{enumerate}
\end{proof}

\begin{thm}\label{thm:dfp2}
Assume $\norm{\delta_D}_s \leq 1/50$, and
take $H := \mu F''_*(\check{s})$ and
$H_1 := H + \frac{xx^\top}{\iprod{s}{x}}
 - \frac{Hss^\top H}{\iprod{s}{Hs}}$.  Then
$$
\norm{\frac{\delta_P \delta_P^\top}
             {\iprod{\delta_D}{\delta_P}}
     - \frac{H_1 \delta_D \delta_D^\top H_1}
             {\iprod{\delta_D}{H_1 \delta_D}}}_s^*
\leq 1.797089 \mu ||\delta_D||_s.
$$
\end{thm}
\begin{proof}
Write
$$
      \frac{\delta_P \delta_P^\top}{\iprod{\delta_D}{\delta_P}}
    - \frac{H_1 \delta_D \delta_D^\top H_1}
           {\iprod{\delta_D}{H_1 \delta_D}}
  = \frac{1}{\iprod{\delta_D}{\delta_P}}
    \left(\delta_P \delta_P^\top - H_1 \delta_D \delta_D^\top H_1\right)
  + \left(\frac{1}{\iprod{\delta_D}{H_1 \delta_D}}
        - \frac{1}{\iprod{\delta_D}{\delta_P}}\right)
    H_1 \delta_D \delta_D^\top H_1.
$$
Notice that
\begin{eqnarray*}
       \abs{\frac{1}{\iprod{\delta_D}{H_1 \delta_D}}
          - \frac{1}{\iprod{\delta_D}{\delta_P}}}
     = \abs{\frac{\iprod{\delta_D}{H_1 \delta_D - \delta_P}}
                 {\iprod{\delta_D}{H_1 \delta_D}
                  \iprod{\delta_D}{\delta_P}}}
& \leq & \frac{||\delta_D||_s ||H_1 \delta_D - \delta_P||_s^*}
            {|\iprod{\delta_D}{H_1 \delta_D}\iprod{\delta_D}{\delta_P}|}\\
  & \leq & 45178250000/75152930769 \frac{1}{\mu ||\delta_D||_s}.
\end{eqnarray*}
Further, recall that $||H_1 \delta_D||_s^* \leq 1.025836 \mu ||\delta_D||_s.$
Thus, the second term's norm is bounded above by $0.632615 \mu ||\delta_D||_s$.
Using the lower bound on $\iprod{\delta_D}{\delta_P}$ and the
upper bounds on $||H_1 \delta_D - \delta_P||_s^*$ and
$||H_1 \delta_D + \delta_P||_s^*$, we get a bound on the first term's
norm of
$$
       \frac{0.542139 \cdot 2.051454}{0.955087} \mu ||\delta_D||_s
  \leq 1.164474 \mu ||\delta_D||_s.
$$
Adding the bounds on the two terms together gives the advertised bound.
\end{proof}

\begin{thm}
\label{thm:7.8}
Assume $\norm{\delta_D}_s \leq 1/50$, and
take $H := \mu F''_*(\check{s})$,
$$
H_1 := H + \frac{xx^\top}{\iprod{s}{x}}
 - \frac{Hss^\top H}{\iprod{s}{Hs}}, \textrm{and}
$$
$$
T^2 := H_1 + \frac{\delta_P \delta_P^\top}
                  {\iprod{\delta_D}{\delta_P}}
          - \frac{H_1 \delta_D \delta_D^\top H_1}
                  {\iprod{\delta_D}{H_1 \delta_D}}.
$$
Then $\norm{T^2 - H} \leq 8.259717 \mu \norm{\delta_D}_s \leq 0.165195 \mu$.
\end{thm}
\begin{proof}
  We consider the two rank-two updates separately;
  Theorem~\ref{thm:dfp1} controls the size of the first
  update and Theorem \ref{thm:dfp2} controls the size of
  the second update.  We simply add the two bounds
  together.
\end{proof}

\begin{cor}\label{t2exists}
If $||\delta_D||_s \leq 1/50$, then there exists a $T \in \SSS^n$
satisfying the following properties:
\begin{itemize}
  \item $T$ is positive definite;
  \item $T^2 s = x$;
  \item $T^2 \tilde{s} = \tilde{x}$;
  \item
    $
              0.814905 \mu F''_*(s)
      \preceq T^2
      \preceq 1.185500 \mu F''_*(s)
    $
  \item
    $
              \frac{0.808093}{\mu} \left(F''(x)\right)^{-1}
      \preceq T^2
      \preceq \frac{1.192311}{\mu} \left(F''(x)\right)^{-1}.
    $
\end{itemize}
That is, $T \in \cT_2(1.237483; x, s)$.
\end{cor}

Note that in the above analysis, we did not utilize the additional flexibility provided by the term $(\mu \tilde{\mu}-1).$
This establishes, in the language of \cite{Tuncel2001}, that
$\xi^*$ is $O(1)$ within a particular neighbourhood of the central
path. Moreover, our specific choice $T_H$ is in $\cT_2(\eta;x,s)$
for $\eta=O(1)$, for every pair $(x,s)$ that is in the same neighbourhood.

Therefore, Theorem 5.1 of \cite{Tuncel2001} implies that a wide range of
potential reduction algorithms (whose iterates are restricted in a neighbourhood
of the central path) have the iteration complexity of
$O\left(\sqrt{\vartheta}\ln\left(1/\epsilon\right)\right)$.
\end{subsection}

\begin{subsection}{Bounds in $v$-space}
  The following lemma is useful to the convergence analysis in the
  next section.\begin{lem}\label{deltaddeltav}
Suppose $x \in \inte(K)$ and $s \in \inte(K^*)$ are such that
$||\delta_D||_s < 1/50.$  Take $T^2$ as in Corollary
\ref{t2exists} and take $T$ to be its self-adjoint positive-definite
 square root.
Let $v := T s = T^{-1} x$ and $\delta_v := T \delta_D$.  Let $z$ be
an arbitrary vector in $v$-space.
Let $x' \in \inte(K)$ and $s' \in \inte(K^*)$; define
$\delta'_D := s' + \mu F'(x')$ and $\delta'_v := T \delta'_D$.
Then,
\begin{enumerate}
\item $\norm{Tz} \leq 1.088807 \sqrt{\mu} \norm{z}_s$;
\item $\norm{z}_s \leq 1.107763 \norm{Tz} / \sqrt{\mu}$;
  \item $||\delta_v|| \leq 0.021777 \sqrt{\mu};$\item $||\delta'_D||_{s'}
      \leq \frac{||\delta'_D||_s}{1 - ||s - s'||_s};$
\item if $||\delta'_v|| \leq 0.006527 \sqrt{\mu}$ and
         $||s - s'||_s \leq 1/25$, then
         $||\delta'_D||_{s'} \leq 0.007533;$
\item if $||\delta'_v|| \leq 0.017330 \sqrt{\mu}$ and
         $||s - s'||_s \leq 1/25$, then
         $||\delta'_D||_{s'} \leq 1/50.$
\end{enumerate}
\end{lem}
\begin{proof}
Recall from Theorem \ref{t2exists} that
\begin{equation}\label{teebounds}
  0.814905 \mu ||z||_s^2 \leq ||T z||^2 \leq 1.185500 \mu ||z||_s^2.
\end{equation}
\begin{enumerate}
\item This is the square root of $\norm{Tz}^2 \leq 1.185500 \mu \norm{z}_s^2$ with a constant rounded up.
\item This is the square root of $\mu \norm{z}_s^2 \leq \frac{200000}{162981} \norm{Tz}^2$ with a constant rounded up.
  \item Take $z = \delta_D$ in part (1) and then use
    $||\delta_D||_s \leq 1/50$; we see
    $$
           ||\delta_v||
         < 1.088807 \sqrt{\mu} ||\delta_D||_s
      \leq 0.021777 \sqrt{\mu},
    $$
    as desired.
\item This is the Dikin ellipsoid bound for comparing the $s'$-norm with the $s$-norm.
\item If $||\delta'_v|| \leq 0.006527 \sqrt{\mu}$, then by part (2)
    $||\delta'_D||_s \leq 0.007231.$  By part (4), then,\\
    $||\delta'_D||_{s'} \leq 0.007379,$ as desired.
\item If $||\delta'_v|| \leq 0.017330 \sqrt{\mu}$, then by part (2)
   $||\delta'_D||_s \leq 0.019198.$  By part (4), then,\\
   $||\delta'_D||_{s'} \leq 0.019998,$ which implies the desired result.
\end{enumerate}
\end{proof}
\end{subsection}

\end{section}

  \begin{section}{Algorithms}\label{sec:algorithms}
    In this section, we assume that some suitable bases have been chosen for
    the underlying spaces and for the sake of concreteness, we write $A$ for
    the underlying matrix representation of $\cA$ etc.  We prove
    $O(\sqrt{\vartheta} \ln \frac{1}{\epsilon})$ iteration complexity bounds on
    variants of the following feasible-start primal-dual interior point
    algorithm with different choices of $\alpha$ and $\gamma$:
    \begin{algorithmic}
      \State Take $k:=0$ and $\left(x^{(0)}, y^{(0)}, s^{(0)}\right)$ to be
        feasible and central (we can also accept approximately central points).
      \While{$\langle x^{(k)}, s^{(k)} \rangle > \epsilon \vartheta$}
        \State Take $T^2$ as in Theorem \ref{thm:7.8}
        \State Select $\gamma \in [0,1]$.
        \State Solve
          $$
            \left(\begin{array}{ccc}
              0 & A^\top & I\\
              A & 0 & 0\\
              I & 0 & T^2\end{array}\right)
            \left(\begin{array}{c} d_x\\d_y\\d_s \end{array}\right)
            =
            \left(\begin{array}{c} 0\\0\\-x^{(k)}-\gamma \mu F'_*(s)
            \end{array}\right).
          $$
        \State Select $\alpha_k \in [0, \infty)$.
        \State $\left(x^{(k+1)}, y^{(k+1)}, s^{(k+1)}\right)
                \gets \left(x^{(k)}, y^{(k)}, s^{(k)}\right) + \alpha (d_x, d_y, d_s)$.
        \State $k \gets k+1$.
      \EndWhile
    \end{algorithmic}

    \begin{lem}
      Let $r_v := -v + \gamma \mu w.$  Then, the system of equations in Line 3
      of the above algorithm imply
      $$
              T^{-1} d_x = \proj(\ker(AT)) r_v,
        \quad T d_s = \proj(\im (A T)^\top) r_v.
      $$
      In particular, $||T^{-1} d_x|| \leq ||r_v||$ and
      $||T d_s|| \leq ||r_v||$.
    \end{lem}
    \begin{proof}
      The third equation ensures that $T^{-1} d_x + T d_s = r_v$.
      The first two equations imply that $T^{-1} d_x$ must lie in $\ker(AT)$
      while $T d_s$ must lie in $\im (AT)^\top.$  Since these two linear
      spaces are orthogonal, the result follows.
    \end{proof}

    The following result does not hint at quadratic convergence.  However, a
    tighter analysis of the low-rank updates showing that the approximation
    error is linear in $||\delta_D||_s$ within the
    $\frac{1}{50}$-neighbourhood would suffice to establish quadratic
    convergence.  This is not hard to do, since the ingredients are already
    given above.  We do not do this here, because quadratic convergence of
    centering is not needed to establish the desired complexity result.

    \begin{lem}
      Suppose $x^{(k)} \in \inte(K)$ and $s^{(k)} \in \inte(K^*)$ define a feasible
      solution.
      If $\gamma_k = 1$ and $\alpha_k = 1$ and
      n
      $||\delta_D^k||_{s^{(k)}} \leq \frac{1}{50}$, then
      \begin{itemize}
        \item $Ax^{(k+1)}=b$ and $A^{\top} y^{(k+1)}+s^{(k+1)} = c$.
        \item $x^{(k+1)} \in \inte(K)$ and $s^{(k+1)} \in \inte(K^*)$.
        \item $||\delta_D^{k+1}||_{s^{(k+1)}} \leq 0.007533$.
        \item $\mu_{k+1} = \mu_k$.
      \end{itemize}
    \end{lem}
    \begin{proof}
      We drop the superscript $k$ when speaking of the $k$th iterate in this
      proof.  The system of linear equations that determine $d_x$, $d_y$, and $d_s$
      guarantee that $d_x \in \ker A$ and $d_s = -A^{\top} d_y$;
      since $Ax^{(k)} = b$ and $A^{\top} y^{(k)} + s^{(k)} = c$, it follows that
      $Ax^{(k+1)} = b$ and $A^{\top} y^{(k+1)} + s^{(k+1)} = c$.

      Notice that, with this choice of $\gamma$,
      $$
        T^{-1} d_x + T d_s = -\delta_v.
      $$
      Since
      $$
             ||d_x||_x
        \leq \frac{\sqrt{1.192311}}{\sqrt{\mu}} ||T^{-1} d_x||
        \leq \frac{\sqrt{1.192311}}{\sqrt{\mu}} ||\delta_v|| \leq 0.024226 < 1,
      $$
      strict primal feasibility is retained.  A similar argument shows that
      strict dual feasibility is retained.

      By Taylor's theorem, there exists an $\bar{x}$ on the segment
      $[x, x + d_x]$ such that $F'(x + d_x) = F'(x) + F''(\bar{x})  d_x.$
      We therefore compute
      \begin{eqnarray*}
             ||T(s + d_s + \mu F'(x + d_x))||
           & = & ||T(\delta_D + d_s + \mu F''(\bar{x}) d_x)||
        \\
        & \leq & ||\delta_v + T d_s + T^{-1} d_x||
             + ||T(T^{-2} - \mu F''(\bar{x})) d_x||.
      \end{eqnarray*}
      The first term is, of course, zero.  However, notice that, by the
      Dikin ellipsoid bound and Theorem \ref{t2exists}
      $$
                \mu F''(\bar{x})
        \preceq 1.299692 T^{-2}
      $$
      and, similarly,
      $$
                \mu F''(\bar{x})
        \succeq 0.798562 T^{-2}
      $$

      We therefore bound
      $$
             \norm{T(T^{-2} - \mu F''(\bar{x})) d_x}
        \leq 0.299692 \norm{T^{-1} d_x} = 0.299692 ||\delta_v||
        < 0.006527 \sqrt{\mu},
      $$
      which implies, by Lemma 7.10 part (5), the advertised bound on the new
      $||\delta_D||_s$.

      As we observed in Section~\ref{sec:prelim}, $\mu$ is
      unchanged by a centering iteration.
    \end{proof}

    \begin{lem}
      If $\gamma_k = 0$ and $\alpha_k = \frac{0.047464}{\sqrt{\vartheta}}$
      and $||\delta_D^k||_{s^{(k)}} \leq 0.007533$, then
      \begin{itemize}
        \item $A x^{(k+1)} = b$ and $A^{\top} y^{(k+1)} + s^{(k+1)} = c$;
        \item $x^{(k+1)} \in \inte(K)$ and $s^{(k+1)} \in \inte(K^*)$;
        \item $\mu_{k+1} \leq (1-\alpha) \mu_k$;
        \item $||\delta_D^{k+1}||_{s^{(k+1)}} \leq \frac{1}{50}$.
      \end{itemize}
    \end{lem}
    \begin{proof}
      We recall that $||\delta_D||_s \leq 0.007533$ means that
      $||\delta_v|| \leq 0.008202 \sqrt{\mu}$.
      Notice that
      $$
             ||d_x||_x
        \leq \frac{1}{\sqrt{0.808093 \mu}} ||T^{-1} d_x||
        \leq \frac{1}{\sqrt{0.808093 \mu}} ||\delta_v||
        \leq 0.009125
      $$
      Consequently, any step with $\alpha < \frac{1}{2\sqrt{\vartheta}}$
      retains strict primal feasibility.  A similar analysis (due to the
      primal-dual symmetry of our set-up) reveals that
      $||d_s||_s \leq 0.008931 \sqrt{\vartheta}$ and hence the dual step retains
      strict dual feasibility for $\alpha$ similarly bounded.  Notice that
      $||\alpha d_s||_s \leq 1/25$; this permits us to use Lemma
      \ref{deltaddeltav} part (6) later.

      As we observed in Section~\ref{sec:prelim}, $\iprod{s(\alpha)}{x(\alpha)}
      =(1-\alpha) \vartheta \mu$.  This establishes the desired reduction in
      $\mu$.

      We compute
      \begin{eqnarray*}
          \norm{T \delta_D^{k+1}}
          &=&\norm{T(s + \alpha d_s + \mu F'(x + \alpha d_x))}\\
          &=&\norm{  T(s + \alpha d_s + \mu F'(x)
                   + \alpha \mu F''(\bar{x}) d_x)}\\
       &\leq&  \norm{\delta_v} + \alpha \norm{T d_s
                   + \mu T F''(\bar{x}) d_x}.
      \end{eqnarray*}
      Let us write
      $$
        E := \mu F''(\bar{x}) - T^{-2}.
      $$
      Then, by Corollary \ref{t2exists},
      $$
        -0.808093 T^{-2} \preceq E \preceq 1.192311 T^{-2}.
      $$
      We thus get an upper bound of
      \begin{eqnarray*}
       \norm{T \delta_D^{k+1}}
       &\leq&  \norm{\delta_v}
             + \alpha \norm{v + T d_s + T^{-1} d_x}
             + \alpha \norm{T E d_x}\\
       &\leq&  \norm{\delta_v}
             + 0
             + 0.192311 \alpha \norm{v}\\
       &\leq&  0.008202 \sqrt{\mu} + 0.192311 \cdot 0.047464 \sqrt{\mu}\\
            & < & 0.017330 \sqrt{\mu}.
      \end{eqnarray*}
      This implies, by Lemma \ref{deltaddeltav} part (6), that
      $\norm{\delta_D^{k+1}}_{s^{(k+1)}} \leq \frac{1}{50}$, as desired.
    \end{proof}

    We immediately have
    \begin{cor}
      Starting from an initial feasible central point, one can alternately
      apply the predictor and corrector steps outlined from the last two
      lemmata and recover an algorithm that takes at most
      $42 \sqrt{\vartheta}$ iterations to reduce $\mu$ by a factor of two.
      In particular, this gives an $O\left(\sqrt{\vartheta} \ln{(1 / \epsilon})\right)$
      bound on the iteration complexity of the algorithm using this choice of $\gamma$.
    \end{cor}

  \end{section}

  \begin{section}{Hyperbolic cone programming: The hyperbolic barriers special case}\label{sec:hyp}
In this section, we assume that we are given access to a hyperbolic polynomial $p$ and $K$
is a corresponding hyperbolicity cone.  As we mentioned earlier, on the one hand, $F(x) := -\ln(p(x))$
is a self-concordant barrier for $K$ with long-step Hessian estimation property. On the other hand,
we do not necessarily have explicit and efficient access to $F_*$ or its derivatives; moreover,
$F_*$ will \emph{not} have the long-step Hessian estimation property, unless $K$ is a symmetric cone.
We will discuss two issues: 
\begin{itemize}
\item
How do we evaluate $F_*$?
\item
Can we compute the primal integral scaling?
\end{itemize}
    \begin{subsection}{Evaluating the dual barrier}
      Given an oracle returning $F(x)$, $F'(x)$, and $F''(x)$ on input $x$,
      we describe an algorithm for approximating $F$'s Fenchel conjugate,
      $F_*$, and discuss its convergence.

      \begin{algorithmic}
        \Loop
          \State $r \gets F'(\hat{x}) + s$
          \If{$||r||_{\hat{x}}^* < \epsilon$}
            \Return $\hat{x}$
          \EndIf
          \State $N \gets -\left[F''(\hat{x})\right]^{-1} r$
          \State $\hat{x} \gets \hat{x} + N$
        \EndLoop
      \end{algorithmic}

      Intuitively, this is steepest descent in the local $\hat{x}$-norm.
      This algorithm is locally quadratically convergent, since the
      dual $s$-norm is well-approximated by the dual $\hat{x}$-norm when
      $-F'(\hat{x})$ is ``close to'' $s$. In particular, one can show that
      if $||r||_{\hat{x}}^* \leq \frac{1}{4}$,
      then the $\hat{x}$-norm of the new residual is at most $\frac{16}{27}$ of
      that of the old.  This implies that the dual $s$-norm of the new residual
      is at most $\frac{64}{81}$ of that of the old residual, ensuring descent.
      The dual $s$-norm is scaled by a factor of at most
      $$
        ||r||_{\hat{x}} / (1 - ||r||_{\hat{x}})^4 \leq 3.2 ||r||_{\hat{x}}
      $$
      in each iteration, establishing local quadratic convergence.

      Note that replacing $s$ with $-F'(\hat{x})$, for some sufficiently
      good approximation $\hat{x}$, can degrade the complementarity gap
      and the measure of centrality $||\delta_P||_x$ and ruins the equation
      $A^{\top} y + s = c$.  Thus one needs to work with an infeasible
      interior-point method or a self-dual embedding technique in the absence
      of an exactly-evaluated dual barrier.  However, it is straightforward to
      bound the local $s$-norm of the increase in residual by
      $||s + F'(\hat{x})||_s$.
    \end{subsection}

    \begin{subsection}{Evaluating the primal integral scaling}
      Given an oracle that can:
      \begin{itemize}
        \item Evaluate $F(x)$, $F'(x)$, and $F''(x)$ for any $x$ in $\inte(K)$, and
        \item Compute, in some explicit form, the univariate polynomial
          $t \mapsto \exp(-F(x + t d))$ for any $x \in \inte(K)$ and $d \in \R^n$,
      \end{itemize}
      we describe how to compute the primal integral scaling
      $$
        \left(\mu \int_0^1 F''(x - t \delta_P) dt\right)^{-1}
      $$
      exactly.  We do not claim that this method is practical or useful;
      in particular, it requires $\vartheta$ evaluations of $F''$.  However,
      we later describe a slightly more practical variant that admits concrete bounds
      on approximation error.

      The method is a straightforward application of the theory of Gaussian
      quadrature.  A reader unfamiliar with Gaussian quadrature might consult
      Section 3.6 of the excellent book by Stoer and Bulirsch
      \cite{StoerBulirsch}.

      Notice that, with $F = -\ln p$, we have
      $$
        F'' = \frac{p'(p')^\top + p p''}{p^2}.
      $$
      The denominator is a polynomial of degree $2\vartheta$ and the numerator
      is a matrix whose entries are polynomials of degree $2\vartheta-2$.

      \begin{thm}
        There exist $\vartheta$ points $r_1, \ldots, r_\vartheta$ in
        $[0,1]$ and associated weights $w_1, \ldots, w_\vartheta$ such that,
        if $q$ is a univariate polynomial of degree less than $2 \vartheta$,
        then
        $$
          \int_0^1 \frac{q(t)}{p^2(x - t \delta_P)} dt
          = \sum_{i=1}^\vartheta w_i q(r_i).
        $$
      \end{thm}
      \begin{proof}
        Notice that $1/p^2(x - t \delta_P)$ is positive and bounded for $t$ in
        $[0,1]$.  In particular, it is measurable, all moments exist and are
        finite, and any polynomial $q$ such that $\int_0^1 q(t)/p^2(x - t
        \delta_P) dt = 0$ is itself zero.  Thus a Gaussian quadrature rule with
        weight function $1/p^2(x - t \delta_P)$ exists.  That is, there exist
        $\vartheta$ points $r_1, \ldots, r_\vartheta$ in
        $[0,1]$ and associated weights $w_1, \ldots, w_\vartheta$ such that,
        for any $\cC^{2\vartheta}$ function $f$,
        $$
          \abs{\int_0^1 \frac{f(t)}{p^2(x - t \delta_P)} dt
             - \sum_{i=1}^\vartheta w_i f(r_i)}
          \leq \frac{f^{(2\vartheta)}(\tau)}{(2\vartheta)!} C
        $$
        for some $0 \leq \tau \leq 1$ and some constant $C \geq 0$ dependent
        on $\vartheta$.

        Take $f = q$; the derivative of order $2\vartheta$ vanishes and hence
        the difference must be zero.
      \end{proof}

      Indeed, the entries of $p' (p')^\top + p p''$ have degree $2\vartheta-2$;
      the primal integral scaling is exactly
      $$
        \left(\mu \sum_{i=1} w_i \left(p'(r_i) (p'(r_i))^\top
                                     + p(r_i) p''(r_i)\right)\right)^{-1}.
      $$

      It may be practical to use a well-known Gaussian quadrature rule
      instead of the one arising from $1/p^2$.  The following theorem considers
      Gauss-Legendre quadrature of fixed order:

      \begin{thm}
        If $1 \leq k$ is an integer and $||\delta_P||_x \leq 1$, then
        $$
          \norm{T_P^{-2} - \mu \sum_{i=1}^k w_i^L F''(x - r_i^L \delta_P)}_x
          \leq \frac{1}{(1 - ||\delta_P||_x)^2}
              \max_{t \in [0,1]} 2 \norm{\delta_P}_{x - t \delta_P}^{2k}
               \norm{F''(x - t \delta_P)}_x^{k/2}
        $$
        where $r_i^L$ are the order-$k$ Gauss-Legendre nodes and
        $w_i^L$ are the associated weights.
      \end{thm}
      \begin{proof}
        Note that
        $$\norm{T_P^{-2} - \mu \sum_{i=1}^k w_i^L F''(x - r_i^L \delta_P)}_x
          = \max_{||h|| \leq 1}
            \abs{T_P^{-2}[h,h] - \mu\sum_{i=1}^k w_i^L F''(x - r_i^L \delta_P)[h, h]}.
        $$
        Thus let $h$ be the point at which this maximum is attained.  By
        the error bound for Gaussian quadrature (\cite{StoerBulirsch}, Theorem
        3.6.24),
        \begin{equation}\label{eq:legendrebound}
          \abs{T_P^{-2}[h,h] - \mu\sum_{i=1}^k w_i^L F''(x - r_i^L \delta_P)[h, h]}
        \leq \max_{t \in [0,1]} \abs{\frac{F^{(2k+2)}(x - t \delta_P)[\delta_P, \ldots, \delta_P, h, h] \iprod{p_k}{p_k}}{(2k)!}},
        \end{equation}
        where $p_k$ is the $k$th Legendre polynomial.

        A theorem of G\"{u}ler (\cite{Guler1997}, Theorem 4.2),
        together with the results from Appendix 1 of Nesterov and Nemirovskii's
        book \cite{NN1994}, shows that
        \begin{equation}\label{eq:legendrederivnorm}
          \abs{F^{(2k+2)}(x - t \delta_P)[\delta_P, \ldots, \delta_P, h, h]}
        \leq (2k+1)! \norm{\delta_P}_{x - t \delta_P}^{2k}
            \norm{h}_{x - t \delta_P}^2.
        \end{equation}
        Using self-concordance, we bound
        $$
             \norm{h}_{x - t \delta_P}
        \leq \frac{1}{1 - t ||\delta_P||_x} \norm{h}_x
           = \frac{1}{1 - t ||\delta_P||_x}.
        $$
        The squared $L_2$ norm of the $k$th Legendre polynomial is famously
        $2 / (2k+1)$.  Substituting these and \eqref{eq:legendrederivnorm}
        into \eqref{eq:legendrebound}, we get the advertised bound.
      \end{proof}
    \end{subsection}
  \end{section}

  \begin{section}{Conclusions and future work}
    We presented a new primal-dual scaling map based on a line integral for
    convex programming where only a $\vartheta$-LHSCB is supplied.  We derived
    some properties of this scaling, notably that it points to the richness
    of potential primal-dual local metrics, enriches the connection of
    primal-dual interior-point methods to Riemannian geometry.  We presented
    a new analysis of low-rank updates of \cite{Tuncel2001}
    showing that,  if one is close to the central path and one begins with a certain
    approximation to the integral scaling, the low-rank update has small norm
    close to the central path.  Some steps of this analysis were peculiar to
    the particular approximation chosen; we leave it to future work to
    generalise the analysis to something depending more directly on the
    approximation error.   We presented a generalization of the Mizuno-Todd-Ye
    predictor-corrector scheme that uses the above tools and showed that it matches the
    current best, worst-case iteration complexity of $O(\sqrt{\vartheta} \ln(1/\epsilon))$,
    of the special case of symmetric cone programming.

    We presented an algorithm for computing an approximation to the conjugate
    barrier given an oracle that computes the primal barrier and discussed
    some bounds that, within the context of an infeasible-start interior-point method
    or a self-dual embedding technique (see \cite{YTM1994,NTY1999}), do not degrade
    the worst-case iteration complexity.

    We presented two techniques based on Gaussian quadrature for evaluating the
    new primal-dual scaling map for hyperbolic barrier functions; one exact,
    and one with bounded approximation error.  Again, we leave to future work
    the problem of tying such an approximation error bound to a bound on the
    magnitude of the necessary low-rank update to the approximation.
  \end{section}

  \addcontentsline{toc}{chapter}{References}
  \bibliographystyle{plain}
  \bibliography{interior}

\begin{thebibliography}{10}

\bibitem{BGLS2001}
Heinz~H. Bauschke, Osman G{\"u}ler, Adrian~S. Lewis, and Hristo~S. Sendov.
\newblock Hyperbolic polynomials and convex analysis.
\newblock {\em Canad. J. Math.}, 53(3):470--488, 2001.

\bibitem{Chua2007}
Chek~Beng Chua.
\newblock The primal-dual second-order cone approximations algorithm for
  symmetric cone programming.
\newblock {\em Found. Comput. Math.}, 7(3):271--302, 2007.

\bibitem{Chua2009}
Chek~Beng Chua.
\newblock A {$T$}-algebraic approach to primal-dual interior-point algorithms.
\newblock {\em SIAM J. Optim.}, 20(1):503--523, 2009.

\bibitem{DS1983}
John~E. Dennis, Jr. and Robert~B. Schnabel.
\newblock {\em Numerical methods for unconstrained optimization and nonlinear
  equations}.
\newblock Prentice Hall Series in Computational Mathematics. Prentice Hall,
  Inc., Englewood Cliffs, NJ, 1983.

\bibitem{Guler1997}
Osman G{\"u}ler.
\newblock Hyperbolic polynomials and interior point methods for convex
  programming.
\newblock {\em Math. Oper. Res.}, 22(2):350--377, 1997.

\bibitem{Hauser2004}
Raphael Hauser.
\newblock The {N}esterov-{T}odd direction and its relation to weighted analytic
  centers.
\newblock {\em Found. Comput. Math.}, 4(1):1--40, 2004.

\bibitem{HauserGuler2002}
Raphael~A. Hauser and Osman G{\"u}ler.
\newblock Self-scaled barrier functions on symmetric cones and their
  classification.
\newblock {\em Found. Comput. Math.}, 2(2):121--143, 2002.

\bibitem{HauserLim2002}
Raphael~A. Hauser and Yongdo Lim.
\newblock Self-scaled barriers for irreducible symmetric cones.
\newblock {\em SIAM J. Optim.}, 12(3):715--723, 2002.

\bibitem{JRT1996}
B.~Jansen, C.~Roos, and T.~Terlaky.
\newblock A polynomial primal-dual {D}ikin-type algorithm for linear
  programming.
\newblock {\em Math. Oper. Res.}, 21(2):341--353, 1996.

\bibitem{Karmarkar1984}
N.~Karmarkar.
\newblock A new polynomial-time algorithm for linear programming.
\newblock {\em Combinatorica}, 4(4):373--395, 1984.

\bibitem{KMNY1991}
M.~Kojima, N.~Megiddo, T.~Noma, and A.~Yoshise.
\newblock {\em A unified approach to interior point algorithms for linear
  complementarity problems}, volume 538 of {\em Lecture Notes in Computer
  Science}.
\newblock Springer-Verlag, Berlin, 1991.

\bibitem{KMY1991}
Masakazu Kojima, Shinji Mizuno, and Akiko Yoshise.
\newblock An {$O(\sqrt n\,L)$} iteration potential reduction algorithm for
  linear complementarity problems.
\newblock {\em Math. Programming}, 50(3, (Ser. A)):331--342, 1991.

\bibitem{LimSIOPT2011}
Yongdo Lim.
\newblock Maximum-volume symmetric gauge ball problem on the convex cone of
  positive definite matrices and convexity of optimal sets.
\newblock {\em SIAM J. Optim.}, 21(4):1275--1288, 2011.

\bibitem{MTY1993}
Shinji Mizuno, Michael~J. Todd, and Yinyu Ye.
\newblock On adaptive-step primal-dual interior-point algorithms for linear
  programming.
\newblock {\em Math. Oper. Res.}, 18(4):964--981, 1993.

\bibitem{Molnar2009a}
Lajos Moln{\'a}r.
\newblock Maps preserving the geometric mean of positive operators.
\newblock {\em Proc. Amer. Math. Soc.}, 137(5):1763--1770, 2009.

\bibitem{MAR1990}
Renato D.~C. Monteiro, Ilan Adler, and Mauricio G.~C. Resende.
\newblock A polynomial-time primal-dual affine scaling algorithm for linear and
  convex quadratic programming and its power series extension.
\newblock {\em Math. Oper. Res.}, 15(2):191--214, 1990.

\bibitem{Myklebust-Tuncel-Talk-2013}
Tor Myklebust and Levent Tun\c{c}el.
\newblock Hyperbolic cone programming: Structure and interior-point algorithms.
\newblock Talk presented at the SIAM Applied Algebraic Geometry Conference,
  Fort Collins, CO, USA, 2013.

\bibitem{NemTun2005}
Arkadi Nemirovski and Levent Tun{\c{c}}el.
\newblock ``{C}one-free'' primal-dual path-following and potential-reduction
  polynomial time interior-point methods.
\newblock {\em Math. Program.}, 102(2, Ser. A):261--294, 2005.

\bibitem{Nesterov2004}
Yurii Nesterov.
\newblock {\em Introductory lectures on convex optimization}, volume~87 of {\em
  Applied Optimization}.
\newblock Kluwer Academic Publishers, Boston, MA, 2004.
\newblock A basic course.

\bibitem{Nesterov2008}
Yurii Nesterov.
\newblock Parabolic target space and primal-dual interior-point methods.
\newblock {\em Discrete Appl. Math.}, 156(11):2079--2100, 2008.

\bibitem{Nesterov2012}
Yurii Nesterov.
\newblock Towards non-symmetric conic optimization.
\newblock {\em Optim. Methods Softw.}, 27(4-5):893--917, 2012.

\bibitem{NN1998}
Yurii Nesterov and Arkadi Nemirovski.
\newblock Multi-parameter surfaces of analytic centers and long-step
  surface-following interior point methods.
\newblock {\em Math. Oper. Res.}, 23(1):1--38, 1998.

\bibitem{NN2008}
Yurii Nesterov and Arkadi Nemirovski.
\newblock Primal central paths and {R}iemannian distances for convex sets.
\newblock {\em Found. Comput. Math.}, 8(5):533--560, 2008.

\bibitem{NN1994}
Yurii Nesterov and Arkadii Nemirovskii.
\newblock {\em Interior-point polynomial algorithms in convex programming},
  volume~13 of {\em SIAM Studies in Applied Mathematics}.
\newblock Society for Industrial and Applied Mathematics (SIAM), Philadelphia,
  PA, 1994.

\bibitem{NT1997}
Yurii Nesterov and Michael~J. Todd.
\newblock Self-scaled barriers and interior-point methods for convex
  programming.
\newblock {\em Math. Oper. Res.}, 22(1):1--42, 1997.

\bibitem{NT1998}
Yurii Nesterov and Michael~J. Todd.
\newblock Primal-dual interior-point methods for self-scaled cones.
\newblock {\em SIAM J. Optim.}, 8(2):324--364, 1998.

\bibitem{NT2002}
Yurii Nesterov and Michael~J. Todd.
\newblock On the {R}iemannian geometry defined by self-concordant barriers and
  interior-point methods.
\newblock {\em Found. Comput. Math.}, 2(4):333--361, 2002.

\bibitem{NTY1999}
Yurii Nesterov, Michael~J. Todd, and Yinyu Ye.
\newblock Infeasible-start primal-dual methods and infeasibility detectors for
  nonlinear programming problems.
\newblock {\em Math. Program.}, 84:227--267, 1999.

\bibitem{NesTun2009}
Yurii Nesterov and Levent Tun{\c{c}}el.
\newblock Local superlinear convergence of polynomial-time interior-point
  methods for conic optimization problems.
\newblock {\em CORE Discussion Paper 2009/72}, 2009; (revised: 2014.

\bibitem{Renegar1988}
James Renegar.
\newblock A polynomial-time algorithm, based on {N}ewton's method, for linear
  programming.
\newblock {\em Math. Programming}, 40(1, (Ser. A)):59--93, 1988.

\bibitem{Renegar2001}
James Renegar.
\newblock {\em A mathematical view of interior-point methods in convex
  optimization}.
\newblock MPS/SIAM Series on Optimization. Society for Industrial and Applied
  Mathematics (SIAM), Philadelphia, PA; Mathematical Programming Society (MPS),
  Philadelphia, PA, 2001.

\bibitem{Renegar2006}
James Renegar.
\newblock Hyperbolic programs, and their derivative relaxations.
\newblock {\em Found. Comput. Math.}, 6(1):59--79, 2006.

\bibitem{Renegar2013}
James Renegar.
\newblock Central swaths: a generalization of the central path.
\newblock {\em Found. Comput. Math.}, 13(3):405--454, 2013.

\bibitem{Renegar-Talk-2013}
James Renegar.
\newblock Primal-dual algorithms for optimization over hyperbolicity cones.
\newblock Talk presented at the SIAM Applied Algebraic Geometry Conference,
  Fort Collins, CO, USA, 2013.

\bibitem{RS2014}
James Renegar and Mutiara Sondjaja.
\newblock A polynomial-time affine-scaling method for semidefinite and
  hyperbolic programming.
\newblock arXiv:1410.6734, 2014.

\bibitem{Rockafellar1970}
R.~Tyrrell Rockafellar.
\newblock {\em Convex analysis}.
\newblock Princeton Mathematical Series, No. 28. Princeton University Press,
  Princeton, N.J., 1970.

\bibitem{Schmieta2000}
Stefan~H. Schmieta.
\newblock Complete classification of self-scaled barrier functions.
\newblock Tech. Report, Dept. of IEOR, Columbia Univ., NY, USA, 2000.

\bibitem{StoerBulirsch}
Josef Stoer, Roland Bulirsch, Richard~H. Bartels, Walter Gautschi, and
  Christoph Witzgall.
\newblock {\em Introduction to numerical analysis}.
\newblock Texts in applied mathematics. Springer, New York, 2002.

\bibitem{SZ1996}
Jos~F. Sturm and Shuzhong Zhang.
\newblock Symmetric primal-dual path-following algorithms for semidefinite
  programming.
\newblock In {\em Proceedings of the {S}tieltjes {W}orkshop on {H}igh
  {P}erformance {O}ptimization {T}echniques ({HPOPT} '96) ({D}elft)},
  volume~29, pages 301--315, 1999.

\bibitem{Tuncel1998}
Levent Tun{\c{c}}el.
\newblock Primal-dual symmetry and scale invariance of interior-point
  algorithms for convex optimization.
\newblock {\em Math. Oper. Res.}, 23(3):708--718, 1998.

\bibitem{Tuncel2001}
Levent Tun{\c{c}}el.
\newblock Generalization of primal-dual interior-point methods to convex
  optimization problems in conic form.
\newblock {\em Found. Comput. Math.}, 1(3):229--254, 2001.

\bibitem{Tuncel2010}
Levent Tun{\c{c}}el.
\newblock {\em Polyhedral and semidefinite programming methods in combinatorial
  optimization}, volume~27 of {\em Fields Institute Monographs}.
\newblock American Mathematical Society, Providence, RI; Fields Institute for
  Research in Mathematical Sciences, Toronto, ON, 2010.

\bibitem{Waterhouse1989}
William~C. Waterhouse.
\newblock Linear transformations preserving symmetric rank one matrices.
\newblock {\em J. Algebra}, 125(2):502--518, 1989.

\bibitem{Wei2002}
Hua Wei.
\newblock {\em Convergence analysis of generalized primal-dual interior-point
  algorithms for linear optimization}.
\newblock Department of Combinatorics and Optimization, Faculty of Mathematics,
  Waterloo, Ontario, Canada, 2002.
\newblock Thesis (M.Math.)--University of Waterloo.

\bibitem{YTM1994}
Yinyu Ye, Michael~J. Todd, and Shinji Mizuno.
\newblock An {$O(\sqrt{n}L)$}-iteration homogeneous and self-dual linear
  programming algorithm.
\newblock {\em Math. Oper. Res.}, 19(1):53--67, 1994.

\end{thebibliography}
  \nocite{*}
\end{document}